\documentclass[oneside,11pt]{article}
\usepackage[margin=2cm]{geometry} 
\usepackage{amsthm, amsmath, amssymb}
\usepackage[mathlines]{lineno}
\usepackage[dvipsnames]{xcolor}
\usepackage[pagebackref,colorlinks,citecolor=Plum,urlcolor=Periwinkle,linkcolor=DarkOrchid]{hyperref}
\usepackage{graphicx}
\usepackage{enumerate}
\usepackage{nicefrac}
\newtheorem{theorem}{Theorem}[section]
\newtheorem{lemma}[theorem]{Lemma}
\newtheorem{corollary}[theorem]{Corollary}
\newtheorem{proposition}[theorem]{Proposition}
\newtheorem{conjecture}[theorem]{Conjecture}

\theoremstyle{definition}
\newtheorem{definition}[theorem]{Definition}
\newtheorem*{remark}{Remark}

\def\C{\mathcal{C}}
\def\P{\mathbb{P}}

\def\G{\mathcal{G}}

\title{Finding geodesics on graphs using reinforcement learning}
\author{Daniel Kious\thanks{Department of Mathematical Sciences, University of Bath, Claverton Down, BA2 7AY Bath, UK.\newline Email: \texttt{d.kious/c.mailler@bath.ac.uk}} \and C\'ecile Mailler\footnotemark[1] \thanks{CM is grateful to EPSRC for support through the fellowship EP/R022186/1.} \and Bruno Schapira\thanks{Aix-Marseille Universit\'e, CNRS, Centrale Marseille, I2M, UMR 7373, 13453 Marseille, France.\newline Email: \texttt{bruno.schapira@univ-amu.fr}}}
\newcommand{\bs}{\boldsymbol}
\newcommand{\sss}{\ensuremath{\scriptscriptstyle}}

\newcommand{\cec}{}
\newcommand{\dan}{}
\begin{document}
%\linenumbers
\maketitle 

\begin{abstract}
It is well-known in biology that ants are able 
to find shortest paths between their nest and the food 
by successive random explorations,
without any mean of communication other than the pheromones they leave behind them.
This striking phenomenon has been observed experimentally 
and modelled by different mean-field 
reinforcement-learning models in the biology literature.

In this paper, we introduce the first 
probabilistic 
reinforcement-learning model for this phenomenon.
In this model, the ants explore a finite graph in which 
two nodes are distinguished as the nest and the source of food.
The ants perform successive random walks on this graph, 
starting from the nest and stopped when first reaching the food, 
and the transition probabilities of each random walk 
depend on the realizations of all previous walks 
through some dynamic weighting of the graph.
{We discuss different variants of this model based on different reinforcement 
rules and show that slight changes in this reinforcement rule can lead to drastically different outcomes.}

We prove that, in two variants of this model
and when the underlying graph is, respectively, any series-parallel graph and 
a 5-edge non-series-parallel {\it losange} graph, 
the ants indeed {\it eventually find the shortest path(s)} 
between their nest and the food.
Both proofs rely on the electrical network method for random walks on weighted graphs 
and on Rubin's embedding in continuous time.
The proof in the series-parallel cases uses the recursive nature of this family of graphs, 
while the proof in the seemingly-simpler losange case turns out to be quite intricate: 
it relies on a fine analysis of some stochastic approximation, 
and on various couplings with standard and generalised P\'olya urns.
\end{abstract}

\section{Introduction and main results}

\subsection{Context and motivation}
In this paper, we introduce and analyse {two variants of a} stochastic, unsupervised, reinforcement-learning algorithm, which, given as an input a graph in which two nodes are marked, gives as output the shortest path(s) between the two marked nodes. 
This algorithm is inspired by mean-field models introduced in the biology literature as models for the behavior of foraging ants (see, e.g.\ \cite{book_ant,current_ants}): 
it has been widely empirically observed (see, e.g.,~\cite{ants_89, current_ants} for experiments) 
that a colony of ants is able to find shortest paths between their nest and the food.
Unsupervised reinforcement learning is widely 
proposed as a model for this phenomenon in the biology literature. 
Our contribution is to introduce a new probabilistic reinforcement-learning 
model for this phenomenon and prove that, in this model, 
the ants indeed find the shortest path between their nest and the food.

We consider a sequence of random walkers on a finite graph $\mathcal G = (V, \mathcal E)$ with two distinguished nodes $N$ and $F$ (for ``nest'' and ``food'' when the walkers are interpreted as ants).
At the beginning of time, all edges of $\mathcal G$ are given weight~1.
The idea is that the walkers explore the graph from~$N$ to $F$ one after each other, and the weights of the edges are updated after each walker reaches $F$. More precisely, for all $n\geq 1$, the $n$-th walker starts a random walk from $N$ and walks randomly on the graph until it reaches~$F$. At every step, the walker chooses one of the neighboring edges with probability proportional to their weights and crosses the chosen edge to the next vertex. Once the $n$-th walker has reached~$F$, we update the weights of the edges by adding~$1$ to {a subset of the trace of this walker.} {In this paper, we look at two possible rules for the choice of this subset of edges to reinforce:}
\begin{itemize}
\item In the {\it loop-erased} version of the model, we reinforce the loop-erased time-reversed trace of walker~$n$. This corresponds to how a hiker without a map would go back from $F$ to $N$ by walking backwards on their own trace, but avoiding unnecesary loops: when facing a choice between several edges they crossed on their way to $F$, they choose the edge that they crossed the earliest on their way forward.
\item In the {\it geodesic} version of the model, we reinforce the shortest path from~$N$ to $F$ inside the trace of the walker (i.e.\ we only look at the subgraph of all edges that were crossed by this specific walker). The case when there are several shortest paths presents some subtleties, on which we will come back when we will define more formally the model in Subsection~\ref{subsec:model} and when discussing our main results (see Subsection~\ref{sec:discussion}).
\end{itemize}
We call this stochastic process the loop-erased or geodesic ant process.

The interpretation of the model in terms of ants is as follows: (1) the ants only lay pheromones behind them on their way back from the food to the nest, (2) each ant goes back to the nest either following the loop-erasure of their forward trajectory reversed in time (for the loop-erased ant process), or following the shortest path in the subgraph that they have explored on the way forward (for the geodesic ant process), and (3) each ant can sense from the amount of pheromones how many of its predecessors have crossed an edge on their way back to the nest, and crosses each neighboring edge with probability proportional to this number. We conjecture that, following this simple unsupervised reinforcement-learning algorithm, the colony of ants {\it eventually finds the shortest path(s) between the nest and the food}, more precisely, asymptotically when time goes to infinity, a proportion~1 of all ants go from the nest to the food following a geodesic.

The difficulty of our analysis comes from different factors: (i) This is a linear reinforcement model: indeed, each ant chooses the next edge to cross with probability proportional to the number of previous ants that laid pheromones on it on their way back to the nest. Interestingly, the assumption that ants react linearly to pheromones is supported in the biology literature (see, e.g.\ \cite{ants_linear, ants_linsup}). In fact, one can easily find counter-examples that show that the same algorithm with super- or sub-linear reinforcement would not find the shortest path (see Subsection~\ref{sec:discussion}). (ii) The algorithm is a sequence of interacting reinforced random walks, and the reinforcement of the $n$-th random walk depends from the realisations of all previous ones.

Our main contribution is to prove that, as conjectured, the ants indeed find the shortest path if we assume that the underlying graph is either a series-parallel graph (as in~\cite{HamblyJordan}) whose ``source'' is the nest and whose ``sink'' is the source of food, or the 5-edge losange graph of Figure~\ref{fig:losange}. 
Surprisingly, the proof for the 5-edge losange graph is more intricate than the proof for the whole class of series-parallel graphs; we therefore expect that finding a proof that would hold for any underlying graph is a very challenging and interesting problem. Both our proofs rely 
heavily on the electric network method for random walks on graphs (see, e.g.,~\cite{LP} for an introduction to this method), 
and Rubin's embedding in continuous time (first introduced in~\cite{Davis}). 
The proof for series-parallel graphs also uses the inductive nature of this family of graphs; 
the proof for the losange graph relies on the fine analysis of different stochastic approximations 
(see, e.g., \cite{Duflo, Pemantle}).
Interestingly, we show that the losange case 
can be seen as an intricate coupling between two types of P\'olya urns (see, e.g.,~\cite{Pemantle} for a survey); 
a fact that is reminiscent of the proof of Pemantle and Volkov~\cite{PV99} 
of the localisation on five sites with positive probability of the {vertex-reinforced random walk} (see also~\cite{Tarres11,Tarres04}).

\subsection{Mathematical description of the model and statement of the main results} \label{subsec:model}
Let $\mathcal G = (V, E)$ be 
a finite graph with vertex set $V$ and edge set $E$. 
Let $N$ (the nest) and $F$ (the food) be two distinct vertices in $V$. 
{In this paper we consider two versions of the same model, which differ by their reinforcement rules.}  

We define the sequence $({\bf W}(n)=(W_e(n)\colon e\in E))_{n\geq 0}$ recursively as follows:
$W_e(0) = 1$ for all $e\in E$, and, for all $n\geq 1$:
\begin{itemize}
\item We sample a random walk $X^{\sss (n)} = (X_i^{\sss (n)})_{i\geq 0}$ on $\mathcal G$ that starts at $N$, is killed when first reaching $F$, and whose transition probabilities are: for all $i\geq 1$, for all $u, v\in V$, 
\[\mathbb P(X_i^{\sss (n)} = v\mid X_{i-1}^{\sss (n)} = u, {\bf W}(n-1))
= \frac{W_{\{u,v\}}(n-1)\bs 1_{u\sim v}}{\sum_{w\sim u}W_{\{u,w\}}(n-1)},\]
where $\{u,v\}$ is the {(unoriented)} edge between $u$ and $v$, 
and $u\sim v$ if and only if the edge $\{u,v\}$ is in $E$.
\item {Let $\mathcal G^{\sss (n)}$ be the trace of $X^{\sss (n)}$, that is the} subgraph of $\mathcal G$ 
obtained when removing from $\mathcal G$ all edges 
that the random walk $X^{\sss (n)}$ did not cross, 
and choose a path of edges $\gamma_n$ as follows: 
\begin{itemize}
\item {For the loop-erased ant process}, we imagine that the walker goes back from $F$ to $N$
by following its trajectory $X^{\sss (n)}$ backwards and avoiding loops as follows:
when the walker is at a vertex that was visited several times on the way forward, possibly coming from different edges at different times, it chooses to cross the edge that was crossed the earliest on the way forward.
We define $\gamma^{\sss (n)}$ as the set of edges crossed by the walker on its way back to the nest. 
\end{itemize}
\begin{remark}
Note that this construction selects a self-avoiding path between $F$ and $N$, which is in fact the loop-erased version of the backward trajectory. Indeed, if we assume that $X^{\sss (n)} = (X_0^{\sss (n)}=N,X_1^{\sss (n)},\dots,X_{K_n}^{\sss (n)}=F)$, for some $K_n\ge1$, and define the time-reversed trajectory $\overline X^{\sss (n)} = (X_{K_n-i}^{\sss (n)},0\le i\le K_n)$, then, by definition, we have that $\gamma^{\sss (n)}_i=\overline X^{\sss (n)}_{j_i}$ for $0\le i\le k_n$ for some $1\le k_n\le K_n$, where $j_0=0$ and $\gamma^{\sss (n)}_{k_n}=F$, for $0\le i\le k_n-1$, $j_{i+1}=\max\{j{\cec +1}: \overline X^{\sss (n)}_{j}=\overline X^{\sss (n)}_{j_i}\}$. This corresponds to the loop-erasure of $\overline X^{\sss (n)}$, as defined in \cite{LL}.
\end{remark}
\begin{itemize}
\item In the uniform-geodesic version of the model, we define $\gamma^{\sss (n)}$ as the shortest path from $N$ to $F$ in $\mathcal G^{\sss (n)}$; if there are several shortest path, we choose one of them uniformly at random.
\end{itemize}
\item For all $e\in E$, set $W_e(n+1) = W_e(n)+ \bs 1_{e\in\gamma_n}$.
\end{itemize}

\begin{figure}
\begin{center}
\includegraphics[width=10cm,page=1]{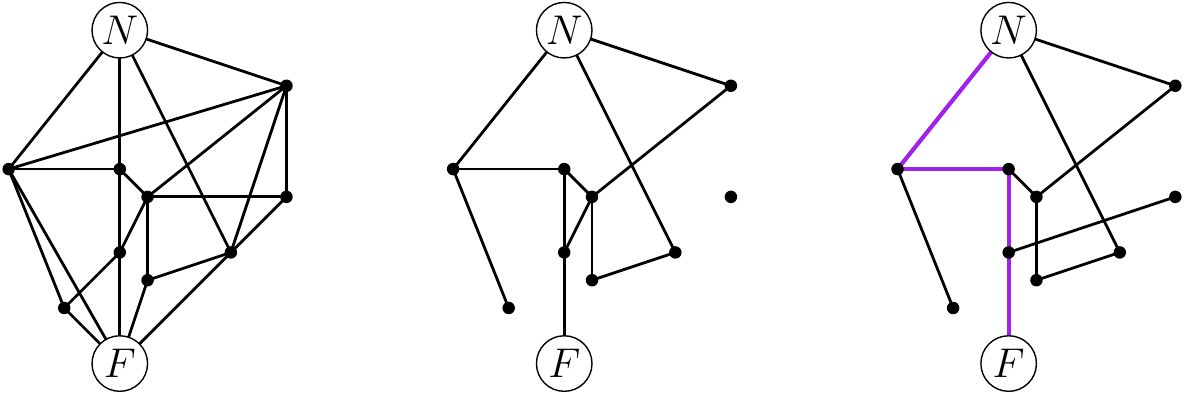}
\end{center}
\caption{First visual aid for the definition of the {uniform-geodesic ant process}. On the left is pictured a graph $\mathcal G$. In the middle is a possible realization of a graph $\mathcal G^{\sss (n)}$, the trace of the $n$-th random walk, and on the right is $\gamma^{\sss (n)}$ the unique geodesic from $N$ to $F$ in $\mathcal G^{\sss (n)}$.}
\label{fig:def1}
\end{figure}
\begin{figure}
\begin{center}
\includegraphics[width=10cm,page=2]{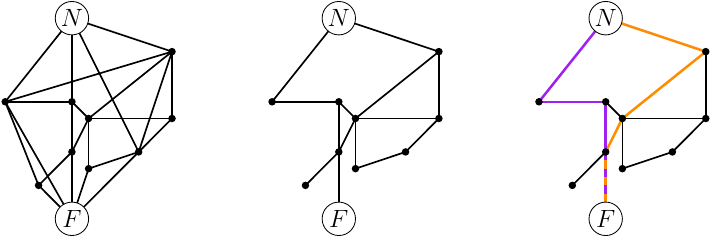}
\end{center}
\caption{Second visual aid for the definition of the {uniform-geodesic ant process}. On the left is pictured a graph $\mathcal G$. In the middle is a possible realization of a graph $\mathcal G^{\sss (n)}$, the trace of the $n$-th random walk. On the right, in orange and purple, are the two geodesics from $N$ to $F$ in $\mathcal G^{\sss (n)}$, and thus the two possible choices for $\gamma^{\sss (n)}$.}
\label{fig:def2}
\end{figure}

We conjecture that 
\begin{conjecture}\label{conj}
Let $\mathcal G = (V, E)$ be any finite graph in which two distinct nodes have been marked  as~$N$ and~$F$. Almost surely when $n\to+\infty$, for all $e\in E$,
\[\frac{W_e(n)}{n} \to \chi_e,\]
where $(\chi_e)_{e\in E}$ is a random vector such that:
\begin{itemize}
\item[{\rm (1)}] {For the loop-erased ant process, $\chi_e \neq 0$ almost surely} {\bf if and only if} the edge $e$ belongs to at least one of the geodesics from $N$ to $F$.
\item[{\rm (2)}] {For the uniform-geodesic ant process,  
$\chi_e\neq 0$ almost surely {\bf only if} the edge $e$ belongs to at least one of the geodesics from $N$ to $F$.}
\end{itemize}
Thus, if there is a unique geodesic $\gamma$ from $N$ to $F$ in $\mathcal G$, then almost surely $\chi_e = \bs 1_{e\in\gamma}$, for all $e\in E$, in the two versions of the model.
\end{conjecture}
This indeed means that the ants {\it eventually find the shortest paths} between their nest and the source of food, because it implies that the probability that the $n$-th ant goes from the nest to the food through a geodesic path converges to~1 when $n\to+\infty$.

The difference between (1) and (2) is that, in the uniform-geodesic ant process, edges that belong to a geodesic may have limiting normalised weight $\chi_e$ that equal zero with positive probability: {\it the ants find at least one of the geodesics, but maybe not all of them.} In Proposition~\ref{prop:ex_Daniel}, we provide an example of a series-parallel graph where $\chi_e=0$ with positive probability for some edge~$e$ on a geodesic path.

\medskip
Our first main contribution is to prove that this conjecture is true for all series-parallel graphs {for the loop-erased ant process}. 
As their name suggests, series-parallel graphs are classical in electricity; in probability theory, they are the object of a famous and and still-open conjecture of Hambly and Jordan~\cite{HamblyJordan}.
They have two distinguished nodes called the ``source'' and the ``sink'', which we can naturally see as the the nest $N$ and the source of food $F$ in our context. 
\begin{figure}
\begin{center}
\includegraphics[width=8cm]{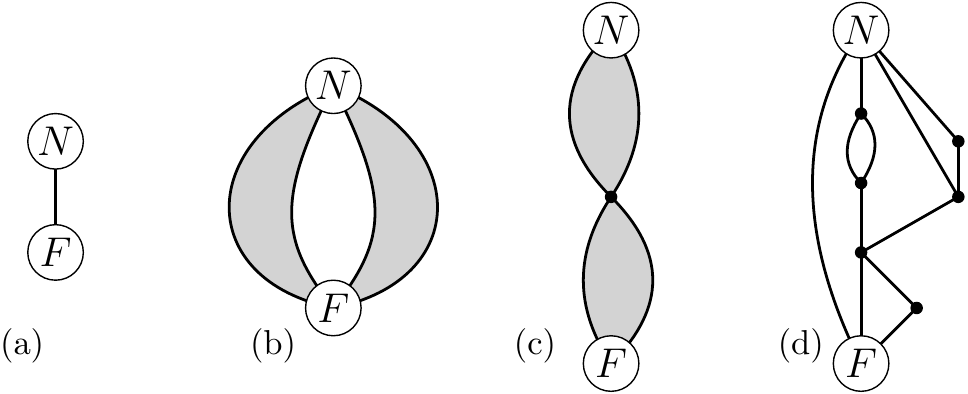}
\end{center}
\caption{The definition of series-parallel graphs: a {\sc sp} graph is either (a) the base case, or (b) two {\sc sp} graphs in parallel, or (c) two {\sc sp} graphs in series. (d) is an example.}
\label{fig:sp}
\end{figure}
\begin{definition}[See Figure~\ref{fig:sp}]
We define series-parallel ({\sc sp}) graphs recursively as follows: a series-parallel graph is
\begin{itemize}
\item either the single-edge graph (graph made of two vertices joined by one edge) with one node marked as the source and the other as the sink,
\item or two series-parallel graphs in series (i.e.\ we merge the sink of the first graph and the source of the second),
\item or two series-parallel graphs in parallel (i.e.\ we merge the two sources and the two sinks).
\end{itemize}
\end{definition}

\begin{theorem}\label{theo.SP}
For any {\sc sp} graph whose source and sink are respectively marked as~$N$ and~$F$, and for the {loop-erased ant process}, Conjecture~\ref{conj} is true, i.e.\ almost surely when $n\to+\infty$, for all $e\in E$,
 \[\frac{W_e(n)}{n} \to \chi_e,\]
where $(\chi_e)_{e\in E}$ is a random vector, such that $\chi_e \neq 0$ almost surely  if and only if the edge $e$ belongs to at least one of the geodesics from $N$ to $F$.  
\end{theorem}

\begin{figure}
\begin{center}
\includegraphics[width=2cm]{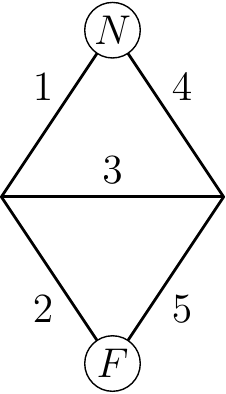}
\caption{The losange graph}
\label{fig:losange}
\end{center}
\end{figure}
Interestingly, the analysis of the loop-erased ant process outside 
the family of series-parallel graphs turns out to be very challenging.
To illustrate this, we consider {one of the simplest non-series-parallel graph one could think of, which is} the 5-edge losange of Figure~\ref{fig:losange}, which we call ``the losange graph'':
even on this simple graph, we are not able to prove convergence of the loop-erased ant process. 
However, we are able to prove convergence of the uniform-geodesic {ant process, which turns out to be simpler in this setting (see the remark before Lemma~\ref{lem:p135p234}).}
 
We number the edges of the losange graph from~1 to~5 as in Figure~\ref{fig:losange}. Our second main result is the following. 
\begin{theorem}\label{theo:losange}
For all $1\leq i\leq 5$ and $n\geq 0$, we denote by $W_i(n)$ ($\forall 1\leq i\leq 5$) the weight of edge number~$i$ after the $n$-th walker has reached the food in the {uniform-geodesic ant process} on the {losange graph}. (Recall that $W_i(0) = 1$, by definition.) 
Almost surely as $n\to+\infty$,
\[\frac{W_i(n)}{n} \to \chi_i,\quad 
\text{ for all }\ 1\le i\le 5,\]
where $(\chi_i)_{1\le i \le 5}$ is a random vector such that almost surely $\chi_1=\chi_2=1-\chi_4=1-\chi_5\in (0,1)$ and $\chi_3=0$.
\end{theorem}

\subsection{Discussion}\label{sec:discussion}

{\bf Discussion on the loop-erased vs.\ uniform-geodesic reinforcement rules:}
While we believe the result on the losange graph is also true for the loop-erased ant process, we think the proof would be 
more involved than with the uniform-geodesic ant process.

The first of three steps in the proof in the uniform-geodesic case is to show that the normalised 
weight of the middle edge (edge number~3) converges to zero, 
and then use this convergence to zero to prove that the speed of convergence to zero is polynomial.
Although proving convergence of the normalised weight of edge~3 would be similar 
(and in fact almost identical) in the loop-erased case,
proving that the speed of convergence is polynomial is, we believe, 
much harder, and could in fact be wrong.
Intuitively, it should not be surprising that the weight of edge~3 could be bigger in the loop-erased than in the uniform-geodesic version of the model: this comes from the fact that reinforcing the edge~3 is more likely at every step in the loop-erased version of the model.
Since the proof in the uniform-geodesic case is already quite involved, 
we leave the case of the loop-erased ant process on the losange open. 

Conversely, the analysis of the uniform-geodesic ant process (and all its variants - see discussion below) 
on series-parallel graphs seems to be a challenging problem, which we also leave for further work. 
In summary, it seems that neither of the two versions of the process is easier to analyse than the other in general, but that this depends on the underlying (family of) graph(s).

\medskip
\begin{figure}
\begin{center}\includegraphics[width=9cm]{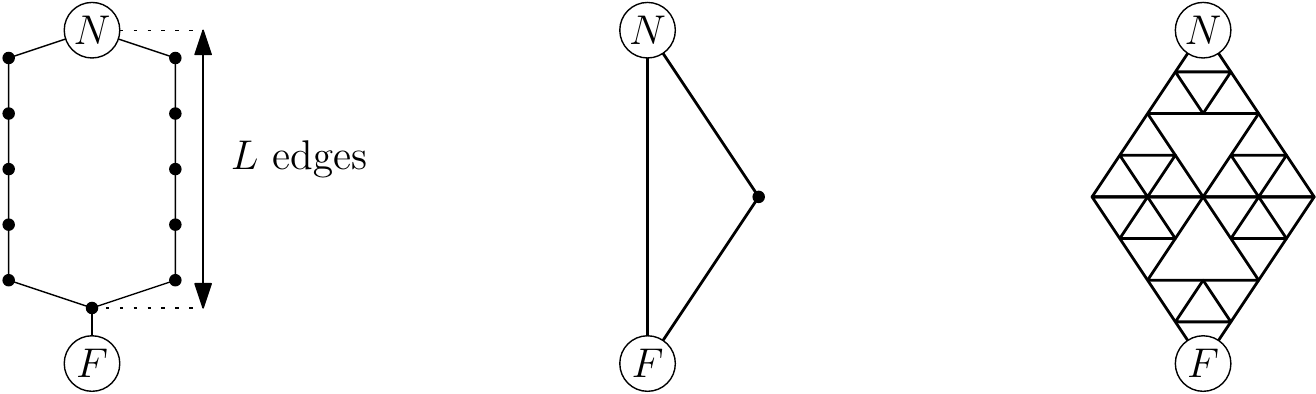}\end{center}
\caption{Graphs used in the discussions of Subsection~\ref{sec:discussion}.}
\label{fig:c_ex}
\end{figure}
{\bf Discussion on the (uniform-)geodesic version of the model:}
First note that on the losange graph, the trace of a walker can contain at most one geodesic, 
and thus the rule of choosing the subset of edges to reinforce uniformly among all geodesics in the trace is irrelevant in this case.
In fact, we believe that the way we choose which shortest path to reinforce when 
there are several in the trace can have a significant impact on the behaviour of the system.

Indeed, we first observe that, in the uniform-geodesic version of the model, 
there could exist an edge that belongs to a geodesic 
between $N$ and $F$ whose normalised weight converges to zero:
\begin{proposition}\label{prop:ex_Daniel}
If $\mathcal G$ is the graph on the left-hand side of Figure~\ref{fig:c_ex}, 
then the uniform-geodesic version of the model satisfies:
there exists~$e\in E$ such that $e$ lies on a geodesic between $N$ and $F$ (in fact, all edges lie on such a geodesic in this graph) 
and, for all $L$ large enough (see Figure~\ref{fig:c_ex} for the definition of~$L$),
\[\mathbb P({W_e(n)}/{n} \to 0)>0.\]
\end{proposition} 
This proposition also holds (with an almost identical proof) when the choice of the geodesic is not uniform as long as any geodesic within the trace is chosen with a probability bounded away from 0.

Another rule for the choice of $\gamma^{\sss (n)}$ when there are several {shortest paths} in $\mathcal G^{\sss (n)}$ is the following: 
Consider $\mathcal G^{\sss (n)}_0$ the subgraph of $\mathcal G^{\sss (n)}$ obtained by removing all  
the edges and vertices that do not belong to any of the shortest paths from $N$ to $F$ in $\mathcal G^{\sss (n)}$. 
As in the loop-erased version of the model, imagine that the walker walks back from $F$ to $N$, 
by only crossing edges from $\mathcal G^{\sss (n)}_0$, 
and, when faced with a choice, choosing the edge it crossed the earliest on the way forward.
Define $\gamma^{\sss (n)}$ as the set of edges crossed by the walker on its way back to the nest.
{We believe that the same conjecture as for the loop-erased version of the model should be true for this version of geodesic ant process.}

\medskip
{\bf Other possible reinforcement rules:}
An alternative reinforcement rule could be to reinforce all edges that the $n$-th walker crossed, i.e.\ all edges in~$\mathcal G^{\sss (n)}$, instead of only reinforce the edges of $\gamma^{\sss (n)}$. 
Intuitively, this would mean that ants lay pheromones on their way to the food instead of laying them on their way back to the nest. 
A mean-field version of this alternative model is also considered in the biology literature (see, e.g.,~\cite{current_ants}). Preliminary work on this alternative reinforcement rule suggests that it could lead to surprisingly different results and that the ants may not always find the shortest path, 
we leave this for further work. 

In this alternative reinforcement rule where ants lay pheromones on their way to the food, one could consider that ants cannot sense from the pheromones laid on an edge how many different ants have crossed this edge, but rather how many times this edge has been crossed by an ant. This would mean that if the~$n$-th ant crossed an edge~$k$ times the weight of this edge is increased by~$k$ when updating the weights after the~$n$-th ant has reached the food. Finally, one could wonder how the results are impacted if the ants are sensitive to their own pheromones, i.e.\ if the weights are updated during the random walks after every steps of the ants, and not after each ant reaches the food. Each ant would then perform a \mbox{(self-)}reinforced random walk that starts on an already-weighted graph. We believe that these variants could lead to different asymptotic behaviours and raise various interesting mathematical challenges.

\medskip
{\bf Discussion on linear vs.\ sub- or super-linear reinforcement:}
As mentioned in the introduction, Conjecture~\ref{conj} would no longer be true if we considered super- or sub-linear reinforcement instead of linear reinforcement.
Indeed, consider the graph in the middle of Figure~\ref{fig:c_ex}, and imagine that all the ants perform weighted random walks on the graph $\mathcal G$, but according to the weights $W_e(n)^\alpha$ ($\forall e\in E$), for some $\alpha>0$.
One can check that if $\alpha>1$ (i.e.\ in the super-linear case), 
then, almost surely, the subset of edges from $E$ such that $\liminf_n W_e(n)/n \neq 0$ is either $\{N,F\}$ or $E\setminus \{\{N,F\}\}$, each with positive probability. Also, if $\alpha<1$ (i.e.\ in the sub-linear case), the subset of all edges from $E$ such that $\liminf_n W_e(n)/n \neq 0$ is almost surely equal to $E$ itself.

\medskip
{\bf Discussion on the underlying graph:}
Theorems~\ref{theo.SP} and~\ref{theo:losange} confirm Conjecture~\ref{conj} in the cases when $\mathcal G$ is a series-parallel graph or when $\mathcal G$ is the losange {graph}, which is the simplest non-series parallel graph. In the proof for series-parallel graph the iterative nature of this family of graph allows us to reason by induction. An iterative family of graphs that builds on the losange example is the ``double Sierpi\'nski gasket'' graph, which consists of two Sierpi\'nski gaskets of the same fractal depth whose bases have been merged (see the right-hand side of Figure~\ref{fig:c_ex} where a double Sierpi\'nski gasket graph of depth 3 is represented). Interestingly, a version of this graph has been considered in the biology literature under the name ``tower of Hanoi'' (see~\cite{current_ants, gasket_ants}).
%\begin{figure}
%\begin{center}
%\includegraphics[width=2cm]{gasket}
%\end{center}
%\caption{The double Sierpi\'nski gasket graph of depth 3.}
%\label{fig:gasket}
%\end{figure}

\medskip
{\bf Other models of path and network formation by reinforcement:}
Our model can be seen as a reinforcement path formation model. The idea is that we start from 
a weighted graph $\mathcal G$ where all edges have the same weight~1, 
and we look at the graph $\mathcal G^{(\infty)}$ of all edges whose {normalised} weight does not tend to zero when time goes to infnity. 
In the langage of Conjecture~\ref{conj}, $\mathcal G^{\infty} = (V, E^{\infty})$ where $e\in E^{\infty}$ if and only if $e\in E$ and $\chi_e>0$.
The fact that $\mathcal G^{\infty}\neq \mathcal G$ means that some path or some network has been selected by the dynamics: in our case, we conjecture (and prove for series parallel graphs or the losange graph) 
that the dynamics selects the shortest paths between the nest and the food.

Other related models of {path} formation by reinforcement exist in the literature: 
for example, Le Goff and Raimond~\cite{LGR} look at a model of non-backtracking vertex-reinforced random walk, 
also inspired from ant behaviour. They show that, in this model, with positive probability, the ant eventually walks along a cycle of finite edges.
This model is very different from ours: the reinforcement is super-linear instead of linear, 
there is one ant as opposed to several ants walking successively in the graph,
the underlying graph is infinite (although locally finite), and there is no nest or food and thus no geodesics involved.

Another related model of network formation is the {\sc warm} model of {\dan v}an der Hofstad, Holmes, Kuznetsov and Ruszel~\cite{WARM}, 
where, at every time step, an edge is chosen at random and its weight increased by one (see also \cite{HKlep}). 
The choice of the edge to reinforce at each step is done according to a two-step procedure that involves super-linear reinforcement.
Van der Hofstad et al.\ {prove that the limiting graph (i.e.\ the graph consisting of all edges whose {normalised} weight does not go to zero) is a linearly-stable equilibrium with positive probability. They conjecture that, if the reinforcement is strong enough, all linearly-stable configuration is a union of trees of diameter at most~3. 
They prove this conjecture in the simple case of a triangle graph, i.e.~the complete graph on three vertices.}

A model of network formation with linear reinforcement is the ``signaling game'' of~\cite{HST11, KT16}, 
where at every time step, ``Nature'' decides which pairs of neighbours are allowed to communicate for this round, and
each vertex chooses a neighbour with probability proportional 
to the number of times they have communicated in the past, 
and {they communicates if they both choose each other and} if Nature allows it.
{In~\cite{KT16}, the authors} show that the limiting graph (consisting of edges between two vertices that 
communicate asymptotically a positive proportion of rounds) is star-shaped {with positive probability}.

\bigskip
{\bf Plan of the paper: } Section~\ref{sec:sp} contains the proof of Theorem~\ref{theo.SP} (i.e.\ the series-parallel case), and Section~\ref{sec:losange} the proof of Theorem~\ref{theo:losange} (i.e.\ the losange case). These two sections can be read independently. 
{Finally we prove Proposition~\ref{prop:ex_Daniel} in Section~\ref{sec:ex_Daniel}. }

%\bigskip 
%Acknowledgments: 

\section{The loop-erased ant process on series-parallel graphs}\label{sec:sp}
In this section, we only consider the loop-erased {ant process}.
We define the size of a graph as its number of edges. For a series-parallel graph $G$, we define its height, which we denote by $h_{\min}(G)$, 
as the length of a shortest path from the source to the sink.

\subsection{Preliminary lemmas}
\noindent We start with two simple observations. The first one is a direct consequence of the definition of series-parallel graphs:
\begin{lemma}\label{decomp.graph}
Let $G$ be a nonempty series-parallel graph. Then, either $G$ is reduced to a single edge (it has size one), or one can find two non-empty series-parallel subgraphs $G_1$ and $G_2$, 
such that $G$ is obtained by merging $G_1$ and $G_2$, either in series or in parallel.    
\end{lemma}

\noindent The second observation is the following:  
\begin{lemma}\label{lem.convex}
Let $\varphi : (0,+\infty)^2\to (0,\infty)$ be the function defined by 
\[\varphi(x,y) = \frac 1{\frac 1x + \frac 1y} \quad\text{ for all }(x,y)\in (0,+\infty)^2.\]
Then, for all $(x, y), (x',y')\in (0,+\infty)^2$, one has
\begin{enumerate}[{\bf (a)}]
\item $\varphi(x +x',y+y') \ge \varphi(x,y) + \varphi(x',y')$, and
\item $\varphi(x+1,y+1) \le \varphi(x,y) +1$.
\end{enumerate}
\end{lemma}

\begin{proof} Since 
\[\varphi\Big(\frac{x+x'}{2},\frac{y+y'}{2}\Big) 
= \frac 12 \cdot \varphi(x+x',y+y'),\]
proving that $\varphi$ is concave is enough to prove~{\bf (a)}. 
A simple calculation shows that 
\[\frac{\partial^2 \varphi }{\partial x^2}(x,y) = \frac{-2y^2}{(x+y)^3},\quad \frac{\partial^2 \varphi }{\partial y^2}(x,y) = \frac{-2x^2}{(x+y)^3},\quad \text{and}\quad \frac{\partial^2 \varphi }{\partial x \partial y}(x,y) = \frac{2xy}{(x+y)^3}.\]
This implies that the Hessian of~$\varphi$ is everywhere non-positive, 
and thus that~$\varphi$ is concave as claimed, which concludes the proof of {\bf (a)}. 

To prove {\bf (b)}, fix $x>0$ and set $H(y):= 1+\varphi(x,y) - \varphi(x+1,y+1)$ for all $y>0$.
Using the definition of $\varphi$, we can calculate 
{\[H'(y) =  \frac1{\big(\frac yx +1\big)^2}-\frac1{\big(\frac{y+1}{x+1} +1\big)^2},\]}
implying that the function $H$ is increasing on $[0,x)$ and decreasing on $(x,+\infty)$. 
Since $H(0)>0$, and $\lim_{y\to \infty} H(y) = 0$, 
it follows that $H(y)>0$ for all positive~$y$, which concludes the proof of {\bf (b)}. 
\end{proof}

\begin{remark} Item (a) has another simple proof in terms of conductances. Indeed one could note that by Rayleigh's monotonicity's principle, putting first $(x,x')$ in parallel and $(y,y')$ in parallel, and then 
putting the two of them in series has a better conductance than first putting $(x,y)$ in series and $(x',y')$ in series, and then putting them in parallel (to go to the first one, we need to add an edge with infinite conductance). 
\end{remark}
\subsection{Our main result in terms of effective conductances}
The main idea to prove Theorem~\ref{theo.SP} is to reason 
in terms of the ``effective conductance'' of the graph.
We interpret the weight of an edge as its ``conductance''
and let $\C_G(n)$ be the effective conductance 
(from the source to the sink) after the $n$-th walk has reached the sink, 
{and simply write $\C_G$ for the initial effective conductance}.
{In order to compute the effective conductance of a series parallel graph, one can use Lemma~\ref{decomp.graph} and the two following rules:}
\begin{itemize}
\item If $G$ is composed of two graphs $G_1$ and $G_2$ merged in parallel, then $\C_G = \C_{G_1} + \C_{G_2}$.
\item If $G$ is composed of two graphs $G_1$ and $G_2$ merged in series, then $\C_G = \varphi\big(\C_{G_1} , \C_{G_2}\big)$.
\end{itemize}
Our main result in terms of effective conductances reads as follows. 
\begin{theorem}\label{th:conductances}
If $G$ is a series-parallel graph, and $\C_G(n)$ is its conductance 
after the $n$-th walker has reached the sink, then, almost surely when $n\to+\infty$,
\[\frac{\C_G(n)}{n}  \to \frac1{h_{min}(G)},\]
where $h_{min}(G)$ is the graph distance between the source and the sink in~$G$.
\end{theorem}

\subsection{Deterministic bounds for the effective conductance of a series-parallel graph after $n$ walks}

The first step towards proving Theorems~\ref{theo.SP} and~\ref{th:conductances} is the following (deterministic) lemma.
\begin{lemma}\label{lem:lowerbound}
{\bf (a)}  Let $G$ be a series-parallel graph with weighted edges and let $\C_G$ be its effective conductance from the source to the sink. Consider a self-avoiding path from the source to the sink of length $L$, and denote by $\C'_G$ the effective conductance {of $G$} after the weights of all edges on this path have been increased by one.
Then,
\[\nicefrac1L\le \C'_G - \C_G \le 1.\]
{\bf (b)} Let $G$ be a series-parallel graph {and consider the loop-erased ant process on $G$}. There exists a constant $C>0$ depending only on~$G$, such that, {almost surely,} 
\[\C_G(n)\le  \frac{n+C}{h_{\min}(G)}, \quad \text{for all }n\ge 0.\]
\end{lemma}
\begin{proof} 
We first prove {\bf (a)} by induction on the size of the graph. 
If $G$ has size one, then the result is immediate {since $\C'_G = \C_G+1$}. 
Now assume that the result holds  
for all series-parallel graphs with size at most~$N$ (for some integer $N\geq 1$) 
and consider a graph~$G$ of size~$N+1$. 
By Lemma \ref{decomp.graph} we know that~$G$ is the merging 
of two non-empty subgraphs~$G_1$ and~$G_2$, either in parallel or in series.
Note that $G_1$ and $G_2$ both have size at most~$N$ 
and thus that the induction hypothesis applies to them.

If $G_1$ and $G_2$ are in parallel, then $\C_G  = \C_{G_1} + \C_{G_2}$. Now since the chosen path is self-avoiding, it either lies entirely in $G_1$ or in $G_2$. Assume for instance 
that it lies in $G_1$: 
using the induction hypothesis, 
we get that $1\ge \C'_{G_1} - \C_{G_1} \ge \nicefrac1L$, 
which concludes the proof since $\C'_G = \C'_{G_1} + \C_{G_2}$.

If $G_1$ and $G_2$ are in series, then first observe that one can write $L= L_1 + L_2$, with $L_i$ the length of the restriction of the path to $G_i$, for $i=1, 2$. 
Then,   
\[
\C_G'   = \frac{1}{\frac 1{\C'_{G_1}} + \frac 1{\C'_{G_2} } }  \ge \frac{1}{\frac 1{\C_{G_1} +\frac 1{ L_1}}  + \frac 1{\C_{G_2}+\frac 1{L_2} }}  \ge  \frac{1}{\frac 1{\C_{G_1}}  + \frac 1{\C_{G_2} }} + \frac{1}{L_1+L_2} = \C_G + \frac{1}{L},
\]
using the induction hypothesis for the first inequality and Lemma~\ref{lem.convex}{\bf (a)} for the second one. This concludes the proof of the lower bound of {\bf (a)}. The proof of the upper bound is entirely similar, using this time Lemma~\ref{lem.convex}{\bf (b)} instead of Lemma~\ref{lem.convex}{\bf (a)}.

Let us now prove {\bf (b)} by induction on the size of the graph again. 
If $G$ has only one edge (which connects the source and the sink), then
$\mathcal{C}_G(n)=1+n$,
which proves the result in this case.
Assume by induction that the upper bound holds for all graphs with at most $N$ edges, and assume that $G$ has $N+1$ edges. By Lemma \ref{decomp.graph}, $G$ consists of two nonempty graphs $G_1$ and $G_2$ which are merged either in parallel or in series, and such that both $G_1$ and $G_2$ have at most $N$ edges. By hypothesis, there exist two constants $C_1$ and $C_2$, such that  for all $n\ge0$,
\begin{equation}\label{induc.G1.G2}
\mathcal{C}_{G_1}(n)\le \frac{n+C_1}{h_{\min}(G_1)},\quad \text{and}\quad \mathcal{C}_{G_2}(n)\le \frac{n+C_2}{h_{\min}(G_2)}.
\end{equation}
If $G_1$ and $G_2$ are in parallel, then $\mathcal{C}_G(n)=\mathcal{C}_{G_1}(n_1)+\mathcal{C}_{G_2}(n-n_1)$, for some (random) integer $0\le n_1\le n$, and the result follows immediately from 
\eqref{induc.G1.G2}, with the constant $C:=C_1+C_2$, using that $h_{\min}(G)= \min(h_{\min}(G_1), h_{\min}(G_2))$. If $G_1$ and $G_2$ are in series, then noting that $h_{\min}(G)= h_{\min}(G_1) + h_{\min}(G_2)$, we get
\[
\mathcal{C}_G(n)=\frac{1}{\frac{1}{\mathcal{C}_{G_1}(n)}+\frac{1}{\mathcal{C}_{G_2}(n)}}\stackrel{\eqref{induc.G1.G2}}{\le}\frac{1}{\frac{h_{\min}(G_1)}{n+C_1}+\frac{h_{\min}(G_2)}{n+C_2} }
\le \frac{n+\max(C_1,C_2)}{h_{\min}(G_1)+h_{\min}(G_2)} = \frac{n+\max(C_1,C_2)}{h_{\min}(G)}. 
\]
This proves the induction step when $G_1$ and $G_2$ are in series, and concludes the proof of the lemma.
\end{proof}
\noindent A consequence of this lemma is that one has the deterministic bounds 
\begin{equation}\label{hmax.ineq}
\frac{n}{h_{\text{max}}(G)}\le \C_G(n)-\C_G(0) \le {\frac{n+C}{h_{\text{min}}(G)}},\quad \text{for all }n\ge 0,
\end{equation}
{for some constant $C>0$ and where $h_{\text{max}}(G)$ is} the length of the longest self-avoiding path from the source to the sink of $G$. In particular, almost surely $\C_G(n) \to \infty$, as $n\to \infty$. Note also that if the ants were always choosing the shortest path, then we would have 
\[\frac{n}{h_{\text{min}}(G)} \le \C_G(n) \le \frac{n+C}{h_{\text{min}}(G)},\]
for some constant $C>0$, for all $n\ge 0$. While the ants usually do not make this optimal choice, we will see that almost surely the asymptotic behavior of the effective conductance of the graph is still of this order (with a weaker control on the error term for the lower bound).  

%%%%%%%%%%%%%%%%%%%%%%%%%%%%%%%%%%%%%%%%%%%%%%%%%%%%%%%%%%%%%%%

\subsection{Bounds for a generalised version of the model}\label{sec:generalised}
In the following, for any series-parallel graph $G$, any (series-parallel) subgraph $H\subseteq G$, and any $n\geq 0$, we let ${\bf W}_{\!H}^G(n)$ 
denote the set of weights on the edges of $G$ after the $n$-th time a path in $H$ has been reinforced. 
We also simply write ${\bf W}_{\!G}(n)$, when $H=G$.

In order to implement an induction argument, we need to consider a generalisation of the loop-erased ant process. 
The reason for this is that we want the law of the process to be stable  under restriction to a subgraph. 
Unfortunately, the loop-erased ant process does not fulfil this: for instance if $G$ is the merging of two subgraphs $G_1$ and $G_2$ in parallel, then when reinforcing a path in $G_1$, an ant on $G$ tends to visit the source less often than an ant restricted to $G_1$. We now explain how we go around this problem.

In the original model on a graph $G$, when the $n$-th ant starts its random walk from $N$, 
it comes back to $N$ a random geometric number of times, say $B_n$, 
and then goes from $N$ to $F$ without returning to $N$. 
We say that the $n$-th ant did $B_n$ {\it unsuccessful excursions} in $G$ (i.e.~going from $N$ to $N$ without hitting $F$), and one {\it successful excursion} (i.e.~going from $N$ to $F$ without returning to $N$). 

In the original model, for all $n\ge1$, $B_n$ is measurable with respect to 
$\mathcal{F}_{n-1}(G):=\sigma({\bf W}_{\!G}(0),\dots,{\bf W}_{\!G}(n-1))$. In the generalised model, we allow $B_n$ and its law to be different and to depend on a larger sigma-field. More precisely, given $\mathcal{F}_{n-1}(G)$ and given some additional integer-valued random variable $B_n$, 
we condition the $n$-th ant on performing $B_n$ unsuccessful excursions before hitting $F$, and then reinforce a path in its range according to the same rule as for the loop-erased ant process, i.e.~we increase by one the weights of the edges along the loop-erasure of the backwards trajectory of the $n$-th ant. 
The only case of interest is when $B_n$ is measurable with respect to some sigma-field of the type $\sigma({\bf W}_{\!G}^{G'}\!(0),\dots,{\bf W}_{\!G}^{G'}\!(n-1))$, where~$G'$ is some series-parallel graph containing $G$; however the proofs of 
the next results work in full generality, without assuming anything on the random variables $B_n$.

For a series-parallel graph $G$, we still let $\mathcal C_G(n)$ denote the effective conductance of graph $G$ after~$n$ walkers have performed their walks and updated the weights in the generalised version of the loop-erased ant process described above.
We set
\begin{equation}\label{eq:def_alpha}
\alpha(G):=\frac{h_{\min}(G)}{h_{\min}(G)+1}.
\end{equation}
The following proposition, together with Lemma~\ref{lem:lowerbound}, implies Theorem~\ref{th:conductances}:

\begin{proposition}\label{prop:good_lowerbound}
Consider a generalised version of the loop-erased ant process on a series-parallel graph~$G$,   
and let $\alpha =\alpha(G)$. There exists a real random variable $K_G$, such that almost surely $K_G$ is finite, and for all $n\ge 1$, 
\begin{itemize}
\item[$(i)$] $\mathcal C_G(n) \ge \frac{n - K_G\cdot n^\alpha}{h_{\min}(G)}$;
\item[$(ii)$] after $n$ steps, the {conditional} probability that the {$(n+1)$-th walk reinforces a geodesic path} is larger than $1- K_G\cdot n^{\alpha-1}$.
\end{itemize}
\end{proposition}

\begin{proof}
We reason by induction on the size of $G$: if $G$ has size~1, then $\C_G(n) = n+1$ almost surely, implying that the result holds. Let us now assume that the result holds for all series-parallel graphs of size at most $N$, and consider a graph $G$ of size $N+1$. 
By Lemma \ref{decomp.graph} we know that $G$ is the merging 
of two nonempty subgraphs $G_1$ and $G_2$, either in parallel or in series. We denote $N_1$ and $N_2$ the sources of $G_1$ and $G_2$, $F_1$ and $F_2$ their sinks.

{\bf Case 1: $G_1$ and $G_2$ are in series.} 
Assume without loss of generality that $G_1$ is on the top of $G_2$ (meaning that the sink $F_1$ of $G_1$ coincides with the source $N_2$ of $G_2$). 
{First note that each ant performing its walk in~$G$ will reinforce one path in $G_1$ and one path in~$G_2$. 
Moreover, by definition of the loop-erasure process, the path that is reinforced in $G_1$ is entirely determined by the trajectory of the walk 
up to its first hitting time of $F_1 = N_2$, 
while the path that is reinforced in $G_2$ 
is entirely determined by the trajectory of the ants after this hitting time of $N_2$. 
As a consequence, 
conditionally on the number of times the walker returns to $N$ before first hitting~$N_2$, 
the laws of the two paths that are reinforced in~$G_1$ and~$G_2$ are independent.}

Furthermore, for each~$n$, the number of unsuccessful excursions in $G_1$ (resp.\ $G_2$) 
that are made by the $n$-th walk before first hitting $N_2$ (resp.\ after first hitting $N_2$) 
is a measurable function of ${\bf W}_{\!G}(n-1)$ and the number $B_n$ of unsuccessful excursions that are prescribed in~$G$.
Therefore, the restrictions of the process to~$G_1$ and~$G_2$ 
are generalised versions of the loop-erased ant process, as defined before Proposition~\ref{prop:good_lowerbound}.
Therefore, we can use the induction hypothesis for $G_1$ and $G_2$:
there exist two random variables $K_1, K_2\in (0,\infty)$, such that with $\alpha_1=\alpha(G_1)$, and $\alpha_2=\alpha(G_2)$, 
\[\C_{G_1}(n) \geq \frac{n-K_1n^{\alpha_1}}{h_{\min}(G_1)}
\quad\text{ and }\quad
\C_{G_2}(n) \geq \frac{n-K_2n^{\alpha_2}}{h_{\min}(G_2)}.\]
If we denote by $\beta = \max(\alpha_1,\alpha_2)$, and by $K=\max(K_1,K_2)$, then
\begin{linenomath}
\begin{align*}
\C_G(n) 
&= \frac1{\frac1{\C_{G_1}(n)} + \frac1{\C_{G_2}(n)}}
\geq \frac1{\frac{h_{\min}(G_1)}{n-K_1n^{\alpha_1}}+\frac{h_{\min}(G_2)}{n-K_2n^{ \alpha_2}}}\\
&\ge \frac{n-Kn^\beta}{h_{\min}(G_1)+h_{\min}(G_2)}
= \frac{n-Kn^\beta}{h_{\min}(G)},
\end{align*}
\end{linenomath}
since $h_{\min}(G) = h_{\min}(G_1)+h_{\min}(G_2)$; 
which concludes the induction argument for Part $(i)$ because, by definition, $\beta \leq \alpha(G)$.

For Part~$(ii)$ we just observe that, by the induction hypothesis {and a union bound}, the {conditional} probability that the $n$-th walker {does not reinforce} a geodesic path is {smaller than $K_1n^{\alpha_1-1}+K_2n^{\alpha_2-1} 
\le Kn^{\beta-1}$},  
which concludes the induction argument in
the case when $G_1$ and $G_2$ are merged in series.

\medskip
{\bf Case 2: $G_1$ and $G_2$ are in parallel.}
We start again by showing that the restrictions of the process on~$G_1$ and~$G_2$ are  generalised versions of the loop-erased model as defined before Proposition~\ref{prop:good_lowerbound}.
For all integers $n$, we denote by $N_i(n)$ the number of times a path in $G_i$ have been reinforced after $n$ ants have performed their walks in~$G$: one has
\begin{equation}\label{CGpar}
\C_G(n)  = \C_{G_1}(N_1(n)) + \C_{G_2}(N_2(n)).
\end{equation} 
We also let $(\tau^{\sss (i)}_k)_{k\geq 1}$ be the random times when the process $N_i$ increases by one, i.e.\ the times when an ant reinforces a path in $G_i$.
For all $n\geq 1$, $k\geq 0$, $i\in\{1, 2\}$, given $\tau_{k-1}^{\sss (i)}$, 
the time to wait until another ant reinforces a path in $G_i$ (i.e.\ $\tau_{k}^{\sss (i)} - \tau_{k-1}^{\sss (i)}$) and the number $B_k^{\sss (i)}$ of unsuccessful excursions made by this ant (the $\tau_k^{\sss (i)}$-th ant) in $G_i$ are both measurable functions of ${\bf W}(\tau_{k-1}^{\sss (i)})$ 
and of the total number of unsuccessful excursions performed in~$G$ by all ants between times $\tau_{k-1}^{\sss (i)}+1$ and $\tau_k^{\sss (i)}$.
Moreover, by definition, given this information, the reinforced path in $G_i$ is chosen by performing $B_k^{\sss (i)}$  independent unsuccessful excursions, plus one additional independent successful excursion, and using the loop-erasure rule. Thus we can use the induction hypothesis for $G_1$ and $G_2$.

In the following, we use the fact that, at any time~$n$, the $(n+1)$-th walker performs its successful excursion in $G_i$ with probability 
$\mathcal C_{G_i}(n)/(\mathcal C_{G_1}(n) + \mathcal C_{G_2}(n))$, for $i=1,2$. 
Indeed, this follows from the fact 
the law of the successful excursion of each ant walking on $G$ is by definition independent of the number of unsuccessful excursions performed by this ant and of their trajectories.
Moreover, for the simple random walk in $G$ (that is if we were considering the original model), 
the probability to reinforce a path in $G_i$ is given by the ratio of the effective conductances, and this happens if and only if the successful excursion belongs to $G_i$.

\medskip
{\bf Case 2.1:} We first assume that $h_{\text{min}}(G_1)=h_{\text{min}}(G_2)$.
Using the induction hypothesis, 
{there exist two random variables $K_1,K_2\in(0,\infty)$ such that, almost surely,}  
\begin{linenomath}
\begin{align*}
\C_G(n) &\geq 
\frac{N_1(n)-K_1N_1(n)^{\alpha}}{h_{\min}(G_1)} 
+ \frac{N_2(n)-K_2N_2(n)^{\alpha}}{h_{\min}(G_2)}\\
&\geq \frac{N_1(n)+N_2(n)}{h_{\min}(G)} - \frac{K(N_1(n)^{\alpha}+N_2(n)^{\alpha})}{h_{\min}(G)}, 
\end{align*}
\end{linenomath}
with $\alpha = \alpha(G)$ (see Equation~\eqref{eq:def_alpha} for the definition of $\alpha(G)$) 
{and $K=K_1+K_2$}. 
This concludes the induction argument for  Part~$(i)$, since by concavity of the map $x\mapsto x^\alpha$, we have
\begin{equation}\label{CGn.concav}
N_1(n)^{\alpha} + N_2(n)^{\alpha}\leq 2^{1-\alpha} n^{\alpha}.
\end{equation}
Concerning Part~$(ii)$, note that using the induction hypothesis, if the $(n+1)$-th walker makes its successful excursion in $G_1$, then the probability that its range contains a geodesic path of $G_1$ is larger than $1-K_1N_1(n)^{\alpha-1}$, and similarly for $G_2$. 
Considering the complement and using a union bound, we deduce that the probability that the $(n+1)$-th walker reinforces a geodesic path of $G$ is at least
\begin{linenomath}\begin{align*}
1-K_1N_1(n)^{\alpha-1}-K_2N_2(n)^{\alpha-1} \ge 1- Kn^{\alpha - 1},
 \end{align*}\end{linenomath}
which concludes the proof of  the induction argument in the case when $h_{\text{min}}(G_1)=h_{\text{min}}(G_2)$.

\medskip
{\bf Case 2.2:} 
We now assume that $h_{\text{min}}(G_1)\neq h_{\text{min}}(G_2)$, and without loss of generality 
$h_{\text{min}}(G_1)<h_{\text{min}}(G_2)$, which implies $\alpha(G) = \alpha(G_1)$ (see Equation~\eqref{eq:def_alpha} for the definition of $\alpha(G)$).
Using the induction hypothesis, we have that there exists a random variable $K_1\in(0,\infty)$,   such that
\begin{equation}\label{CG1.1}
\mathcal{C}_{G_1}(N_1(n))\ge \frac{N_1(n)}{h_{\min}(G_1)} \cdot (1-K_1N_1(n)^{\alpha-1}).
\end{equation}
For small values of $N_1(n)$, this lower bound can be negative 
(recall that, by definition, $\alpha=\alpha(G) <1$; see Equation~\eqref{eq:def_alpha});
a better lower bound for small values of $N_1(n)$ is given by 
\begin{equation}\label{CG1.2}
\mathcal{C}_{G_1}(N_1(n))  \ge \mathcal C_{G_1}(0).
\end{equation}
By Lemma~\ref{lem:lowerbound}{\bf (b)}, there exists a constant $C_2>0$ (only depending on $G_2$), such that
\begin{equation}\label{CG2.1}
\mathcal{C}_{G_2}(n-N_1(n))\le \frac{n-N_1(n)+ C_2}{h_{\min}(G_2)} \le \frac{n-N_1(n)+C_2}{h_{\min}(G_1)+1},
\end{equation}
because, by assumption, $h_{\min}(G_2)\geq h_{\min}(G_1)+1$.
For all $b>0$, we define the function $\varphi_b$ such that, for all $i\ge0$,
\begin{equation}\label{eq:def_phi}
\varphi_b(i) := \max\left(\C_{G_1}(0), \frac{i-b i^{\alpha}}{h_{\min}(G_1)}\right).
\end{equation}
We also define the function $\psi$ such that, for all $i\ge0$,
\begin{equation}\label{eq:def_psi}
\psi(i):= \frac{\alpha(i+C_2)}{h_{\min}(G_1)}.
\end{equation}
By Equations~\eqref{CG1.1}, \eqref{CG1.2}, and \eqref{CG2.1} 
we get that the probability $p_n$ that the $(n+1)$-th reinforces a geodesic path of $G_1$, conditionally on ${\bf W}_{\!G}(n)$, satisfies
\begin{equation}\label{pn}
p_n = \frac{\mathcal{C}_{G_1}(N_1(n))}{\mathcal{C}_{G_1}(N_1(n))+{\mathcal{C}_{G_2}(n-N_1(n))}} 
 \ge  \frac{\varphi_{K_1}(N_1(n))}{\varphi_{K_1}(N_1(n)) + \psi(n-N_1(n))}, 
\end{equation}
for all $n\ge 0$.
We now prove that, almost surely, there exists a finite random variable $K>0$, such that 
\begin{equation}\label{eq:N1}
N_1(n)\ge n - Kn^\alpha, \quad \text{ for all }n\ge 1.
\end{equation} 
This is enough to concludes the proofs of the induction step for both Parts~$(i)$ and~$(ii)$. 
Indeed, on the one hand, we get that, for all $n\ge 1$,
\[
\mathcal{C}_{G}(n)\ge \mathcal{C}_{G_1}(N_1(n))
\ge  \frac{N_1(n)(1-K_1N_1(n)^{\alpha-1})}{h_{\min}(G_1)}
\ge \frac{n-(K+ K_1)n^\alpha}{h_{\min}(G)},
\]
which concludes the proof of the induction step of Part~$(i)$.
And, on the other and, using \eqref{hmax.ineq}, we get that the probability of not reinforcing a geodesic path in~$G$ is smaller than
\[K_1N_1(n)^{\alpha-1}+(1-p_n)\le (K_1+h_{\max}(G)K )n^{\alpha-1},\]
which concludes the proof of the induction step of Part~$(ii)$.
Therefore, to conclude the proof, it only remains to prove Equation~\eqref{eq:N1}.

If $K_1$ was a fixed constant, the conclusion would come by simply analysing the generalised urn process associated to $\varphi_{K_1}$ 
and $\psi$.
But here, $K_1$ is a random variable that depends on the whole history of the process. To go around this issue, we are going to define 
a family of generalised P\'olya urns, and couple all of them with the process $(N_1(n))_{n\geq 0}$, in such a way that almost surely 
$(N_1(n))_{n\geq 0}$ will dominate at least one of those urns. 
To be more precise, for all $b>0$, we define the Markov process $(R_n^b)_{n\ge 0}$,  by $R_0^b = 0$, and for all $n\ge 0$,
\begin{equation}\label{eq:def_Rb}
q_n^b:=\mathbb P(R_{n+1}^b = R_n^b+1 \mid R_n^b) = 1- \mathbb P(R_{n+1}^b = R_n^b \mid R_n^b) =  \frac{\varphi_b(R_n^b)}{\varphi_{b}(R_n^b) + \psi(n-R_n^b)},
\end{equation}
where $\varphi_b$ and $\psi$ are defined in Equations~\eqref{eq:def_phi} and~\eqref{eq:def_psi} respectively.
We now fix some $b>0$ 
and show that there exists an almost surely finite random variable $C_b$, such that 
\begin{equation}\label{eq:Rn_Rubin}
n-R_n^b \le C_b n^\alpha, \qquad \text{for all }n\ge 0. 
\end{equation}
To prove Equation~\eqref{eq:Rn_Rubin}, it is convenient to use 
Rubin's algorithm, which was introduced in 
Davis's paper on reinforced random walks \cite{Davis}. 
Consider $\{\xi_i^1\}_{i\ge 0}$ and $\{\xi_i^2\}_{i\ge 0}$ 
two independent sequences of independent mean-one exponential random variables, 
and define, for all $n\ge 1$, 
\[T_n^1 : = \sum_{k=0}^{n-1} \frac{\xi_k^1}{\varphi_b(k)},\quad \text{and}\quad T_n^2 : = \sum_{k=0}^{n-1} \frac{\xi_k^2}{\psi(k)}.\] 
Set also $T_0^1=T_0^2=0$ and, for all $t>0$ (see Figure~\ref{fig:rubin.1}),  
\[\tau^1(t):= \sup \{n\ge 0  : T_n^1\le t\},  \quad \text{and}\quad \tau^2(t):= \sup \{n \ge 0 : T_n^2\le t\}.\] 
\begin{figure}
\begin{center}
\includegraphics[width=12cm]{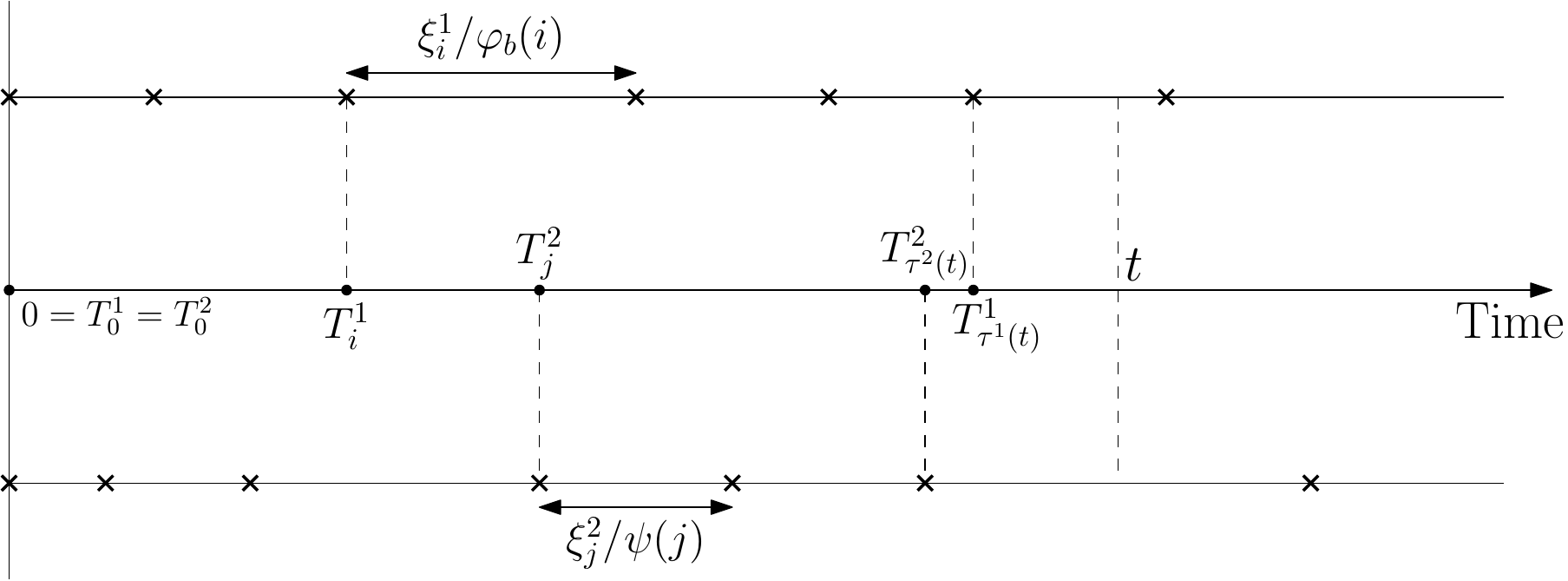}
\end{center}
\caption{Rubin's construction for the proof of Proposition~\ref{prop:good_lowerbound} (Case 2.2 in the proof). 
On the top line, the intervals between crosses are the $\xi^1_i/\varphi_b(i)$ and, similarly, on the bottom line,
the intervals between crosses are the $\xi^2_i/\psi(i)$. On the middle line, we show how the $T_i^1$'s and $T_i^2$'s are defined as the partial sums of these interval lengths and how $\tau^1_t$ and $\tau^2_t$ are defined for a given time~$t>0$.}
\label{fig:rubin.1}
\end{figure}
It follows from standard properties of independent exponential random variables that, for any $t>0$, 
conditionally on the fact that $\tau^1(t) = n_1$, and $\tau^2(t) = n_2$,
the probability $q_{n_1+n_2}^b$ that $R_{n_1+n_2}^b$ increases by one at the next step 
is also equal to the probability of $T_{n_1+1}^1$ being smaller than $T_{n_2+1}^2$.

As a consequence if we let $t_n=\inf\{t\ge 0 : \tau^1(t)+\tau^2(t)\ge n\}$, then the process 
$(\tau^1(t_n))_{n\ge 0}$ has the same law as $(R_n^b)_{n\ge 0}$. 
Note that, since they are bounded in $L^2$, the series  
\[\sum_{k=0}^\infty \frac{\xi_k^1- 1}{\varphi_b(k)} \quad \text{and}\quad \sum_{k=0}^\infty \frac{\xi_k^2- 1}{\psi(k)}\]
converge almost surely. In particular 
\begin{equation}\label{Ti.Rubin}
T_n^1 = \sum_{k=0}^{n-1} \frac 1{\varphi_b(k)} +\mathcal O(1) =\log n + \mathcal O(1) , \quad \text{and}\quad T_n^2 = \sum_{k=0}^{n-1} \frac 1{\psi(k)} +\mathcal O(1) = \frac 1\alpha \log n + \mathcal O(1),
\end{equation}
where the $\mathcal O(1)$ are almost surely bounded. 
Moreover, by definition, one has 
\[\sup_{n\ge 0} |T^1_{\tau^1(t_n)} - T^2_{\tau^2(t_n)} |\le \sup_{n\ge 0} \ \max\Big(\frac{\xi_n^1}{\varphi_b(n)},\frac{\xi_n^2}{\psi(n)}\Big),\]
from which it follows that 
\[T^1_{\tau^1(t_n)} = T^2_{\tau^2(t_n)} + \mathcal O(1),\] 
where $\mathcal O(1)$ stands for an almost surely finite random variable. 
Together with~\eqref{Ti.Rubin}, this entails $ \tau^2(t_n)\le  C_bn^\alpha$, for all $n\ge 0$, and some almost surely finite random variable $C_b$, which concludes the proof of Equation~\eqref{eq:Rn_Rubin}.

To conclude the proof of Equation \eqref{eq:N1}, we only need to
couple the family of processes $(R_n^b)_{n\geq 0}$, $b>0$, 
with $(N_1(n))_{n\geq 0}$ so that, almost surely, there exists $K>0$
such that $N_1(n)\ge R_n^K$ and $p_n\ge q_n^K$, for all $n\geq 0$. 
To do this coupling, we use a sequence $(U_n)_{n\ge1}$ of i.i.d.~uniform random variables on $[0,1]$, independent of everything else. 
We start the processes so that 
$N_1(0)=0$ and $R^b_0=0$ for all $b>0$. 
Then, at each time step $n\ge0$, set $N_1(n+1)=N_1(n)+1$ if and only if $p_n\ge U_{n+1}$ and, similarly for all $b>0$, $R^b_{n+1}=R^b_n+1$ if and only if $q^b_n\ge U_{n+1}$. 
By induction on $n$, we can prove that, in this coupling, 
for all $b\ge K_1$, for all $n\geq 1$, $N_1(n)\ge R^b_n$. 
Indeed, first note that, by Equation~\eqref{pn}, for all $b\ge K_1$, $N_1(n)\ge R_n^b$ implies $p_n\ge q_n^b$.
Moreover, if $N_1(n)\ge R_n^b$ and $p_n\ge q_n^b$, then $N_1(n+1) \ge R_{n+1}^b$, 
which concludes the proof by induction: we get that, 
for all $b\ge K_1$, $N_1(n)\ge R^b_n\ge n-C_b n^\alpha$. 
This concludes the proof of Equation \eqref{eq:N1}, thus the proof of the induction step in Case~2.2, 
and thus the proof of Proposition~\ref{prop:good_lowerbound} altogether.
\end{proof}

\subsection{Proof of Theorem~\ref{theo.SP}}\label{sub:proof_th_sp}
Since the original model is a particular case of the generalised model of Section~\ref{sec:generalised},
it is enough to prove that Theorem~\ref{theo.SP} holds in the generalised model.

By~Proposition~\ref{prop:good_lowerbound}, for any edge $e$ that is not contained in a geodesic path, one has $W_e(n)/ n\to 0$, when $n\to +\infty$. Thus it only remains to show that, for every edge $e$ that lies on a geodesic path, $W_e(n)/ n$ converges to some 
random variable $\chi_e$, which is almost surely non-zero.

The proof is done by induction on the size of $G$. If $G$ has size one, the result is straightforward. We now assume that the result holds for all series-parallel graphs of size at most $N$, and  consider a series-parallel graph $G$ of size $N+1$. 
Once again, by Lemma \ref{decomp.graph} we know that $G$ is the merging of two non-empty subgraphs~$G_1$ and~$G_2$. If $G_1$ and $G_2$ are in series, then the result for $G$ follows immediately from the induction hypothesis. 

Let us now assume that $G_1$ and $G_2$ are merged in parallel. If $h_{\text{min}}(G_1)\neq h_{\text{min}}(G_2)$, and for instance if $h_{\text{min}}(G_1)< h_{\text{min}}(G_2)$, then the proof in the previous subsection shows that a fraction $1-o(1)$ of the ants chooses a path in $G_1$, and then the result follows from the induction hypothesis.

If $h_{\text{min}}(G_1)= h_{\text{min}}(G_2)$, 
we first show that $\liminf N_i(n)/n>0$, almost surely for all $i\in \{1,2\}$.  
To do this, we use again Rubin's construction; 
the argument is very similar to the one given in Case 2.2 of the proof of Proposition~\ref{prop:good_lowerbound}. 
We only briefly indicate how to adapt the proof to show that $\liminf N_i(n)/n>0$ in the present case.
{\dan We aim at coupling the process $(N_1(n))_{n\geq 0}$ with a family of processes $(R^b_n)_{n\geq 0}$, $b>0$.}
We define $\varphi_b$ as in Equation~\eqref{eq:def_phi}
and set $\psi (i) = (i+C_2)/h_{\text{min}}(G_1)$ for all integers $i$ (compare with Equation~\eqref{eq:def_psi}). We then define $R_n^b$ as in Equation~\eqref{eq:def_Rb}.
One can show that, on the one hand, for any $b>0$, there exists a {\dan random variable} $c_b>0$, such that {\dan almost surely} for all $n\geq 1$, $R_n^b  \ge c_b n$. 
And, on the other hand, {\dan there exists a random $b>0$ such that $N_1(n) \ge R_n^b$ for all $n\ge0$ almost surely. Hence,} we deduce that 
almost surely $\liminf N_1(n)/n>0$, as claimed. 
In other words, almost surely, a positive fraction of the ants chooses a path in $G_1$, and by symmetry the same holds for $G_2$.

We now show that $N_1(n)/n$ converges almost surely when $n$ tends to infinity. 
To do this, we show that $X(n):=N_1(n)/n$ {\cec is} a stochastic approximation. 
Indeed, we have, for all $n\ge 1$, 
\[X(n+1) = X(n) + \frac{\Delta M_n + h_n}{n+1},\]
where $\Delta M_n = N_1(n+1)- N_1(n) - p_n$, 
with $p_n$ as defined in~\eqref{pn}, and $h_n: = p_n - X(n)$.
Iterating the above equation, we get that, for all $n\ge 1$, 
\[X(n+1) = X(1) + \sum_{k=1}^n \frac{\Delta M_k + h_k}{k}.\] 
Note that, by definition, the martingale increment $\Delta M_k$ is bounded by~$1$ in absolute value, and thus the martingale $\sum_{k=1}^n \Delta M_k / k$, is bounded in $L^2$, and hence 
almost surely convergent. 
Using the definition of $p_n$ (see Equation~\eqref{pn}), together with {\dan Lemma~\ref{lem:lowerbound}{\bf (b)}, Proposition~\ref{prop:good_lowerbound}, and the fact that $\liminf N_i(n)/n>0$, for $i=1,2$, 
one can show that, almost surely when $n$ tends to infinity,  
$h_n= \mathcal O(n^{\alpha(G_1)-1})$, where we recall that, by definition (see Equation~\eqref{eq:def_alpha}), $\alpha(G_1)<1$}. 
This implies that the sum $\sum_{k=1}^n h_k/k$ is almost surely convergent, 
and thus that $X(n)$ converges almost surely, as claimed. 
Together with the induction hypothesis applied to $G_1$ and $G_2$, 
this allows us to conclude the induction step for that last case 
($G_1$ and $G_2$ merged in parallel and $h_{\min}(G_1) = h_{\min}(G_2)$).

Altogether, this concludes the proof of Theorem \ref{theo.SP}.

%%%%%%%%%%%%%%%%%%%%%%%%%%%%%%%%%%%%%%%%%%%%%%%%%%%%%%%%%%%%%%%%%%%

\section{The geodesic ant process on the losange graph}\label{sec:losange}
We prove here Theorem \ref{theo:losange} concerning the losange graph; in this section, we thus only consider the (uniform-)geodesic version of the model (as discussed in Section~\ref{sec:discussion}, the rule about how to choose the geodesic to reinforce when there are several in the trace of the walker is irrelevant here since the trace of a walker can only contain one geodesic). 
The proof relies primarily on the fact that the sequence of weights is the solution of a certain stochastic recursion formula, 
which we state in Lemma~\ref{lem:setE} below. 

Recall Figure~\ref{fig:losange} of the losange graph, and define for $n\ge 0$, 
\begin{equation} \label{def.W}
{\bf W}(n) := (W_1(n), W_2(n), W_3(n),W_4(n),W_5(n)),\quad \text{and}\quad \hat{{\bf W}}(n)=\frac{{\bf W}(n)}{n+2}, 
\end{equation}
where $W_i(n)$ denotes the weight of edge $i$ after $n$ walkers (or ants) have reached the food.
Then for $w=(w_1,\dots,w_5) \in [0,1]^5$, denote by $p_{12}(w)$ the probability that a walker reinforces edges $1$ and $2$, when the weights
of the five edges of the losange graph are respectively $w_1,\dots,w_5$. Define similarly $p_{135}(w)$, $p_{234}(w)$, and $p_{45}(w)$, and set 
\begin{equation}\label{def.F}
F(w) := p_{12}(w)(1,1,0,0,0)+p_{135}(w)(1,0,1,0,1)
+p_{45}(w)(0,0,0,1,1)+p_{234}(w)(0,1,1,1,0)- w. 
\end{equation}
Lemma~\ref{lem:setE} expresses the fact that the whole study of the process $({\bf W}(n))_{n\ge 0}$ takes place in the subset of $[0,1]^5$, defined as 
\begin{equation}\label{set.E}
\mathcal E:=\left\{(w_1,w_2,w_3,w_4,w_5)\in [0,1]^5 : 
\begin{array}{lll}
w_1+ w_4 = 1, & \text{and} & w_2+w_5=1\\
|w_1-w_2|\le w_3  & \text{and}& |w_5-w_4|\le w_3 \\
w_1+w_2\ge w_3 & \text{and} & w_4+w_5\ge w_3
\end{array}
\right\}.
\end{equation}
Let us briefly explain the restrictions above. Note that each walk can only reinforce one of the following sets of edges:
\begin{enumerate}
\item[$(i)$] edge 1 and edge 2;
\item[$(ii)$] edge 4 and edge 5;
\item[$(iii)$] edge 1, edge 3 and edge 5;
\item[$(iv)$] edge 4, edge 3 and edge 2.
\end{enumerate}
One can see above that, at each round, precisely one of edge 1 or edge 4 is reinforced and precisely one of edge 2 or edge 5 is reinforced. 
Hence, we have that $w_1+w_4=w_2+w_5=1$. Next, the only cases where edge 1 is reinforced but not edge 2, or edge 2 is reinforced but not edge 1
are in scenarios $(iii)$ and $(iv)$, in which cases edge 3 is reinforced. Therefore, $|w_1-w_2|\le w_3$, and by symmetry $|w_4-w_5|\le w_3$. 
Finally, again using $(iii)$ and $(iv)$, every time edge 3 is reinforced edge 1 or edge 2 is reinforced. Therefore $w_1+w_2\ge w_3$ and by symmetry $w_4+w_5\ge w_3$.

Using further the definition of the ant process, we obtain Lemma~\ref{lem:setE} below. We use now the shorthand notation $\mathbb E_n$ to denote 
the conditional expectation with respect to the sigma-field $\mathcal F_n$ (where $(\mathcal F_n)_{n\ge 0}$ is the natural filtration of the process).
\begin{lemma}\label{lem:setE}
For all $n\geq 0$, $\hat{{\bf W}}(n)\in \mathcal E$. Furthermore,  
\begin{equation}\label{eq:algo_sto}
\hat{{\bf W}}(n+1) = \hat {{\bf W}}(n) + \frac1{n+3}\big(F(\hat{{\bf W}}(n))+ \Delta \bs M(n+1)\big),
\end{equation}
where $\Delta \bs M(n+1) = Y(n+1) - \mathbb E_n[Y(n+1)]$, and 
 $Y(n+1) := {\bf W}(n+1) - {\bf W}(n)$.
\end{lemma}

As mentioned in the introduction, this losange case can be seen as an intricate coupling between a biased urn (the ants that reinforce edge~3 versus all others, i.e.\ $W_3(n)$ vs.\ $n-W_3(n)$) and a standard P\'olya urn (the ants that reinforce edges~1 and~2 vs.\ the ants that reinforce edges~4 and~5).
In Subsection~\ref{subsec:W3} we treat the first urn by proving that $W_3(n)/n$ converges to $0$ almost surely, at a polynomial speed.
The ``P\'olya'' part is treated in two additional steps:
In Subsection~\ref{subsec:convW} we show 
that $\hat{{\bf W}}(n)$ converges almost surely to some limit in $[0,1]^5$,
and, in Subsection~\ref{subsec:nondegenerate}, 
we prove that the limit is non-degenerate, 
in the sense that it does not charge the extremal 
points~$(1, 1, 0, 0, 0)$ and $(0,0,0, 1, 1)$.
In terms of the ants, this means that the ants {\it find both geodesics} and not just one of them.
Interestingly, ruling out these extremal cases is the most delicate part of the proof.

\subsection{On the convergence of $W_3(n)/n$ to $0$}\label{subsec:W3} 
In this section, we prove here the following result. 
\begin{proposition}\label{prop:W3}
Almost surely, as $n\to +\infty$, one has $W_3(n)/n \to 0$. More precisely, there exists $\alpha \in (0,1)$, such that almost surely, 
\[\lim_{n\to \infty} \frac{W_3(n)}{n^\alpha} = 0.\]
 \end{proposition}

The first idea of the proof is to compare $W_3(n)$ with the number of red balls in a two-colour Friedman-like urn defined as follows:

\begin{lemma}\label{lem:Flike_urn}
We define a Markov process $(R_n)_{n\geq 0}$ as follows: first $R_0=1$, and 
for all $n\geq 0$, we set $R_{n+1} = R_n + A_{n+1}$, 
where
\[\mathbb P(A_{n+1} = 1 \mid R_n) =  1- \mathbb P(A_{n+1} = 0 \mid R_n)
:= \frac{R_n}{n+2} \cdot\frac{(\frac{R_n}{n+2})^2+\frac12}{\frac{R_n}{n+2}+\frac12}.\]
Then almost surely when $n\to+\infty$, we have 
$R_n/n  \to 0$.  
\end{lemma}
\begin{proof}
Let us define $Z_n:=R_n/(n+2)$, for all $n\ge 0$. We use stochastic approximation: by definition, we have that, for all $n\geq 0$,
\[Z_{n+1}=\frac{R_{n+1}}{n+3}
= \frac{R_n + A_{n+1}}{n+3}
= \frac{R_n}{n+2} \cdot \frac{n+2}{n+3} + \frac{A_{n+1}}{n+3}
= Z_n + \frac1{n+3}\big(A_{n+1} - Z_n\big).\]
For $n\geq 0$, set $\Delta M_{n+1} = A_{n+1}-\mathbb E \big[A_{n+1} \mid R_n\big]$.
By definition of the model, we have
\[\mathbb E \big[A_{n+1} \mid R_n\big]
= Z_n \cdot\frac{Z_n^2+\frac12}{Z_n+\frac12},\]
implying that
\[
Z_{n+1}
= Z_n + \frac1{n+3}\left(G(Z_n)+\Delta M_{n+1}\right),\]
where, for all $x\in[0,1]$,
\[G(x) = x \cdot\frac{x^2+\frac12}{x+\frac12}-x.\]
Note that $G(x) \leq 0$ for all $x\in[0,1]$. Thus $(Z_n)_{n\ge 0}$ is a non-negative supermartingale, and converges almost surely. 
Moreover, by definition $|\Delta M_n|\le 1$, for all $n\ge 0$, and thus
the martingale 
\[\widetilde M_n:= \sum_{i=1}^{n-1} \frac{\Delta M_{i+1}}{i+3},\]
converges almost surely, since it is bounded in $L^2$. It follows that the series $\sum G(Z_n)/n$ also converges almost surely, which implies that the limit of $(Z_n)_{n\ge 0}$ is necessarily a zero of $G$, that is either $0$ or $1$. 
To see that $Z_n\to 0$ almost surely, we couple $(Z_n)_{n\geq 0}$ with a P\'olya urn: this coupling is based on the fact that, by definition and because $\frac{x^2+1}{x+1}\leq 1$ for all $x\in [0,1]$, we have
\[\mathbb P(A_{n+1} = 1\mid Z_n)\leq Z_n.\]
Thus if we define a process $(U_n)_{n\geq 0}$ such that $U_0 = Z_0$ and, for all $n\geq 0$,
\[\mathbb P(U_{n+1} = U_n +1\mid U_n) = 1-\mathbb P(U_{n+1} = U_n\mid U_n)= U_n,\]
then $(U_n)_{n\geq 0}$ and $(Z_n)_{n\geq 0}$ can be coupled in a way that $Z_n\leq U_n$ almost surely for all $n\geq 0$. It is known that $U_n \to U$ almost surely when $n\to+\infty$, where $U$ is uniform on $[0,1]$. Thus $Z_n$ cannot converge to~1 and thus converges to 0 almost surely when $n\to+\infty$.
\end{proof}

The next step to prove Proposition~\ref{prop:W3} is to compute the probability that a walker reinforces the middle edge~3. 
Recall the definition~\eqref{set.E} of the set $\mathcal E$.   
\begin{lemma}\label{lem:prob_zigzag}
One has for all $w\in \mathcal E$, 
\[p_{135}(w) =   \frac{w_1w_3w_5}{(w_2+w_3+w_1w_4)(w_4+w_5)+w_2w_3+w_1w_3w_4}.\]
\end{lemma}

\begin{proof} 
We call ``left'' vertex the vertex linked to edges 1, 2 and 3, and ``right'' vertex the vertex between edges 3, 4 and 5. To renforce edges 1, 3 and 5, a walker has to
\begin{enumerate}[(i)]
\item go through edge 1 in its first step,
\item then, from the left vertex, reach the right vertex before going through edge~2,
\item finally, from the right vertex, reach the food before going through edge~2 or~4.
\end{enumerate}
Let us denote by $p_i$, $p_{ii}$ and $p_{iii}$ the respective probabilities of these three events; we thus have $p_{135}(w) = p_i p_{ii}p_{iii}$. First note that
\[p_i = \frac{w_1}{w_1+w_4} = w_1,\]
using for the last equality that $w\in \mathcal E$.
To calculate $p_{ii}$ and $p_{iii}$, we use effective conductances. One can check that $p_{ii}$ is the probability that a random walker starting from the black dot in the left-hand side of Figure~\ref{fig:ii} reached the white dot before reaching one of the crosses.  
In Figure~\ref{fig:ii}, we use the parallel and series formulas for effective conductances to simplify the left-hand side graph into the equivalent (in terms of effective conductances) right-hand side graph. In the right-hand side graph, it is easy to see that the probability to reach the white dot before the cross starting from the black dot is
\[p_{ii} = \frac{w_3}{w_2+w_3+\frac{w_1w_4}{w_1+w_4}} =\frac{w_3}{w_2+w_3+w_1w_4} ,\]
using again that $w\in \mathcal E$ for the last equality. 
\begin{figure}
\begin{center}
\includegraphics[width=12cm]{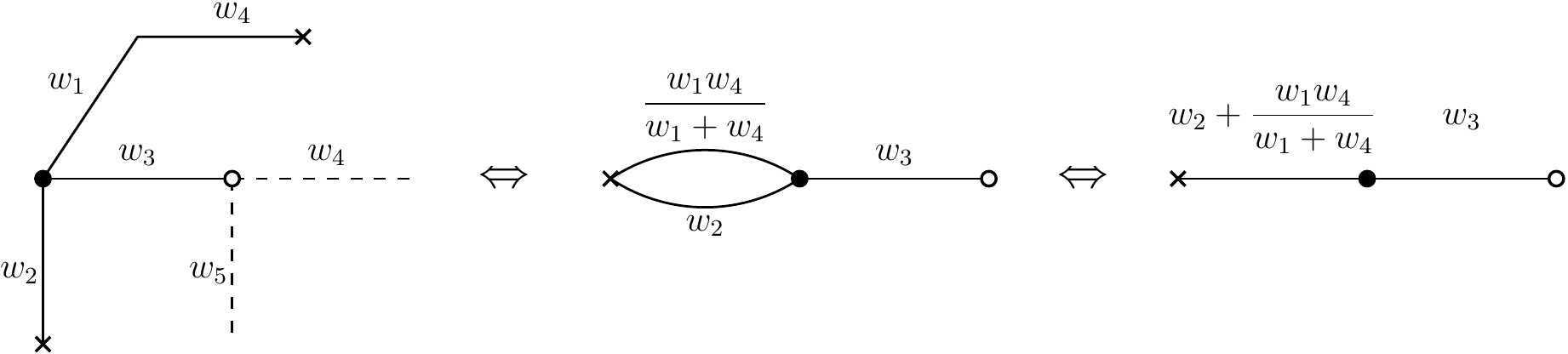}
\end{center}
\caption{Calculation of the probability of $(ii)$, the event that a random walker starting at the black dot reaches the white dot before reaching the crosses. The dashed edges in the left-hand side picture have no effect on the calculation and can be removed. In terms of effective conductances between the black dot and the crosses and the black dot and white dot, these three graphs are equivalent.}
\label{fig:ii}
\end{figure}
Similarly, one can check that $p_{iii}$ is the probability that a walker staring from the black dot in the left-hand side of Figure~\ref{fig:iii} reaches the white dot before reaching one of the crosses. Using the calculation of effective conductances done in Figure~\ref{fig:iii}, we eventually get that
\[p_{iii} = \frac{w_5}{w_4+w_5+\frac{w_2w_3(w_1+w_4)+w_1w_3w_4}{(w_2+w_3)(w_1+w_4)+w_1w_4}} =\frac{w_5}{w_4+w_5+\frac{w_2w_3+w_1w_3w_4}{w_2+w_3+w_1w_4}} ,\]
which concludes the proof, {since, for all $w\in\mathcal E$, $w_1+w_4 = 1$}.
\begin{figure}
\begin{center}
\includegraphics[width=14cm]{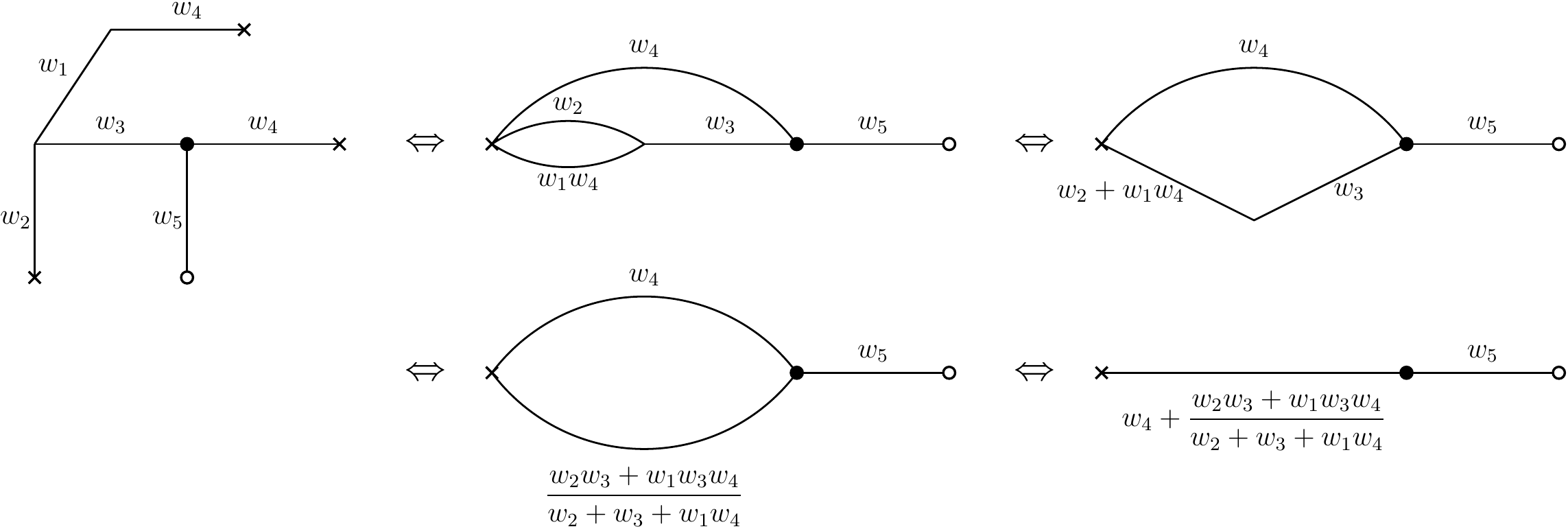}
\end{center}
\caption{Calculation of the probability of $(iii)$ in the case $w_1+w_4=1$, the event that a random walker starting at the black dot reaches the white dot before reaching the crosses. In terms of effective conductances between the black dot and the crosses and the black dot and white dot, these five graphs are equivalent.}
\label{fig:iii}
\end{figure}
\end{proof}

We deduce the following result, proving the first part of Proposition~\ref{prop:W3}. 

\begin{lemma}\label{lem:W3}
One has for all $w\in \mathcal E$, 
\[
p_{135}(w)+p_{234}(w) \le  w_3 \cdot \frac{w_3^2+\frac 12}{w_3+\frac 12}, 
\]
and as a consequence almost surely, 
\[\lim_{n\to \infty} \, \frac{W_3(n)}{n} = 0. \] 
\end{lemma} 
\begin{proof}
The idea is the following: we run the ants walk from time~0, and simultaneously, we consider an urn that contains black and red balls. We call this urn the ``ants urn''. At time zero, we put one black ball and one red ball in the urn, and everytime an ant reaches the food in the ants walk process, we add a ball into the urn: this ball is red if edge number 3 has been reinforced by this ant, black otherwise. The first part of the lemma will show that this urn can be coupled with a Friedman-like urn of Lemma~\ref{lem:Flike_urn} so that there are always more red balls in the Friedman-like urn.

By Lemmas~\ref{lem:setE} and~\ref{lem:prob_zigzag} we have for all $w\in \mathcal E$,  
\begin{equation}\label{eq:p135_utile}
p_{135}(w) 
=  \frac{w_1w_3w_5}{(w_2+w_3+w_1w_4)(w_4+w_5)+w_2w_3+w_1w_3w_4} \le \frac{w_1w_3w_5}{(w_3+w_2)(w_4+w_5) +w_3w_2 }. \end{equation}
Using that $w_3+w_2 \ge w_1$, and $w_2+w_5 = 1$, we deduce
\begin{equation}\label{eq:borne_p135}
p_{135}(w)  \le \frac{w_1w_3w_5}{w_3 + w_1w_4+w_2w_5}.
\end{equation}
By symmetry, we have that
\[p_{234}(w) \leq \frac{w_2w_3w_4}{w_3 + w_1w_4+w_2w_5},\] 
and thus, the probability that the $n$-th walker reinforces edge 3 is at most
\[p_{135}(w)+p_{234}(w) \leq w_3 \cdot \frac{w_1w_5+ w_2w_4}{w_3+w_1w_4 + w_2w_5} .\]
Finally, we note that 
\[w_1w_5 +w_2w_4 = w_1w_4 + w_2w_5 + (w_1-w_2)(w_5-w_4) \le w_1w_4 + w_2w_5 + w_3^2,\] 
which entails 
\[p_{135}(w)+p_{234}(w) \leq w_3 \cdot \frac{w_1w_4+ w_2w_5+w_3^2}{w_1w_4 + w_2w_5+w_3}  
\leq {w_3 \left(1-\frac{w_3 (1-w_3)}{w_1w_4+w_2w_5+w_3}\right)}.\]
Recalling next that, for all $x\in[0,1]$, $x(1-x)\leq \nicefrac14$  and that $w_1+w_4=w_2+w_5=1$, we have that $w_1w_4+w_2w_5\le \nicefrac12$, which implies 
\[
p_{135}(w)+p_{234}(w) \le w_3 \left(1-\frac{w_3(1-w_3)}{w_3+\frac 12}\right)
= w_3 \cdot \frac{w_3^2+\frac 12}{w_3+\frac 12},  
\]
proving the first part of the lemma. Applying this with $w=\hat{{\bf W}}(n)$,
we thus have proved that, at every time step $n$, the probability to add a red ball in the ants-urn is at most the probability to add a red ball in the Friedman-like urn of Lemma~\ref{lem:Flike_urn}. Therefore, the number of red balls in the ants urn (i.e.\ $W_3(n)$) is at most $R_n$ at time~$n$ (for all $n\geq 0$), where $R_n$ is the quantity defined in Lemma~\ref{lem:Flike_urn}. Thus the result follows from Lemma~\ref{lem:Flike_urn}. 
\end{proof}

\begin{remark}
It is interesting to note that, in the loop-erased ant process, one has
\[p_{135} = \frac{w_1w_3w_5}{w_3 + w_1w_4+w_2w_5},\]
to compare with Equation~\eqref{eq:borne_p135}.
This means that Lemma~\ref{lem:W3} holds in this case too.
However, to show almost-sure convergence of $\hat{{\bf W}}(n)$, we need to know that the convergence of $\hat W_3(n)$ to zero has polynomial speed. This is done in the following lemma, whose proof relies on a better bound, using the equality in Equation~\eqref{eq:p135_utile}.
Therefore, the fact that this better bound does not hold in the loop-erased case 
is the reason why we believe that the proof of Conjecture~\ref{conj} in that case is more intricate.
\end{remark}

We now aim at bootstraping the previous result to get a polynomial speed of convergence. For this we will need the following fact.

\begin{lemma}\label{lem:p135p234}
For any $\rho\in (0,{\nicefrac16})$, there exists $\varepsilon>0$ 
such that for any $w\in \mathcal E$ satisfying $w_3\le \varepsilon$, 
\[p_{135}(w) + p_{234}(w) \le (1-\rho)w_3.\]
\end{lemma}
\begin{proof}
By Lemma~\ref{lem:prob_zigzag}, for any $w\in \mathcal E$, 
\begin{equation}\label{eq.p135}
p_{135}(w) = \frac{w_1w_3w_5}{w_3(1+w_4+w_1w_4) + w_2w_4 + w_2w_5 + w_1w_4^2 + w_1w_4w_5 }. 
\end{equation}

Assume that $w_3<\nicefrac14$. 
Let us first prove a lower bound on the denominator of \eqref{eq.p135}.
 This denominator is at least equal to
$w_3(1+w_4) + w_2w_4 + w_2w_5$, and we would like to prove that
\begin{equation}\label{eq:first_ineq_denom}
w_3(1+w_4) + w_2w_4 + w_2w_5\geq -w_3^2 +2w_1w_5.
\end{equation}
Indeed, first using the fact that, for all $w\in\mathcal E$, $w_4\geq w_5-w_3$, we get
\[w_3(1+w_4) + w_2w_4 + w_2w_5
\geq w_3(1+w_4) + w_2(w_5-w_3) + w_2w_5
\geq w_3 (1+w_4-w_2) + 2w_2w_5.
\]
Now, using the facts that, for all $w\in\mathcal E$, 
$w_2\geq w_1-w_3$, $w_4-w_5\ge -w_3$, and $1-w_2 = w_5$, we get that 
\[w_3(1+w_4) + w_2w_4 + w_2w_5
\geq w_3 (w_4+w_5) + 2(w_1-w_3)w_5\geq w_3 (w_4-w_5)+2w_1w_5
\geq -w_3^2 +2w_1w_5,\] 
which concludes the proof of~\eqref{eq:first_ineq_denom}.

Next we distinguish two cases: either $w_2\ge w_1$ or $w_2<w_1$.

$\bullet$ We first treat the case when $w_1\ge w_2$ and, as a consequence, $w_5\ge w_4$.
Plugging Equation~\eqref{eq:first_ineq_denom} into Equation~\eqref{eq.p135}, 
we thus get
\[
p_{135}(w)   
\le \frac{w_1w_3w_5}{2w_1w_5-w_3^2}.
\]
Since $w\in \mathcal E$, we have $w_1+w_2\ge w_3$, which, since $w_1\geq w_2$ implies $w_1\ge \nicefrac{w_3}2$. Similarly, the facts that $w_4+w_5\ge w_3$ and $w_5\geq w_4$ imply that $w_5\ge \nicefrac{w_3}2$. 
Moreover, since $w\in \mathcal E$, we have $w_1+w_4 = 1$, 
and thus either $w_1\ge \nicefrac12$ or $w_4\ge \nicefrac12$. 
If $w_1\ge \nicefrac12$ then we conclude that $w_1w_5\geq \nicefrac{w_5}2\geq \nicefrac{w_3}4$.
If $w_4\ge \nicefrac12$, then $w_5\ge w_4\ge \nicefrac12$, and we also get $w_1w_5\geq \nicefrac{w_3}4$ in this case. Therefore, in both cases ($w_1\ge \nicefrac12$ and $w_4\ge \nicefrac12$), using the fact that $\frac 1{1-x} \le 1+2x$ for all $0\le x\le \nicefrac12$, we get
\[p_{135}(w)\le \frac{w_3}{2(1-2w_3)}
\leq \frac{w_3}{2} + 2w_3^2,\]
as long as $w_3<\nicefrac14$.

$\bullet$ We now treat the case when 
$w_2\ge w_1$, which implies $w_4\ge w_5$. 
In that case, it is straightforward to see that the denominator in~\eqref{eq.p135} is at least $w_2(w_4+w_5)\geq 2w_1w_5$, which implies
$p_{135}(w)\le w_3/2$.

By the two cases above, we have thus proved that, for all $w\in \mathcal E$
such that $w_3\le\nicefrac14$,
\[p_{135}(w)\le \frac{w_3}{2} + 2w_3^2.\]
Note that by symmetry the same inequality holds for $p_{234}(w)$, i.e.\ for all $w\in\mathcal E$,
\begin{equation}\label{eq:bootstrap}
\max(p_{135}(w),p_{234}(w))\le \frac{w_3}{2} + 2w_3^2, 
\end{equation}
 but this is not yet enough to conclude the proof: we need to get a better upper bound by taking into account the terms in the denominator 
 of Equation~\eqref{eq.p135} that we previously neglected.

To do that, we again distinguish two cases: 
first assume that $w_4\ge \nicefrac12$. 
In this case, using the fact that for all $w\in\mathcal E$, $w_4\geq w_5-w_3$, we get
\[w_1w_4^2 \ge w_1w_4w_5 - w_1w_3w_4
\ge \frac 12 w_1w_5 - w_1w_3w_4,\]
and since we assume that $w_4\ge \nicefrac12$, we also get
\[w_1w_4w_5\geq \frac12w_1w_5.\]
Now if in addition $w_5\ge w_4$, one has $w_1\ge w_2$ and thus $w_1\ge w_3/2$, as well as $w_5-w_4\le w_3 \le 4 w_1w_5$. This, together with the last two displays and~\eqref{eq.p135} implies
\[p_{135}(w) \le \frac{w_1w_3w_5}{3w_1w_5 + w_3(w_4-w_5)} 
\leq \frac{w_3}{3}\cdot \frac1{1-\frac{4w_3}3}
\leq  \frac{w_3}{3}\Big(1+\frac{8w_3}3\Big),\]
as long as $w_3\leq \nicefrac38$.
We thus get that, for all $w\in\mathcal E$ such that $w_3\leq \nicefrac38$, and $w_4\ge \nicefrac 12$,
\[p_{135}(w) \le \frac{w_3}3 + w_3^2.\]
We now need to treat the case when $w_4\le \nicefrac12$. 
In that case, $w_1\ge \nicefrac12$ and we get, by symmetry, 
\[p_{234}(w) \le  \frac{w_3}3 + w_3^2.\]
In both cases ($w_4\geq \nicefrac12$ and $w_1\geq \nicefrac12$), 
using Equation~\eqref{eq:bootstrap}, we get
\[p_{135}(w)+p_{234}(w) \leq \Big(\frac12+\frac13\Big)w_3 + 3w_3^2
= \frac{5w_3}{6}+3w_3^2,\]
as long as $w_3\leq \nicefrac14$, 
and the lemma follows.
\end{proof}

\begin{lemma}\label{lem:W3poldec}
Almost surely, for any $\alpha>5/6$, 
\[\lim_{n\to \infty} \frac{W_3(n)}{n^{\alpha}} = 0.\]
\end{lemma}
\begin{proof}
Fix $\alpha>5/6$, and set $Z_n := n^{-\alpha}\cdot W_3(n)$. Using Equation~\eqref{eq:algo_sto}, we get, for all $n\ge 1$, 
\[
Z_{n+1} = Z_n \cdot \left(1-\frac 1{n+1}\right)^{\alpha} + \frac{W_3(n+1)-W_3(n)}{(n+1)^{\alpha}}  = Z_n + \frac{r_n}{n+1}+ \frac{\Delta M_3(n+1)}{(n+1)^\alpha},
\]
with $r_n=(n+1)^{1-\alpha}\mathbb E_n Y_3(n+1) - \alpha Z_n + \mathcal O(Z_n/n)$,
almost surely when $n\to+\infty$.
Recall that $Y(n+1) = {\bf W}(n+1)-{\bf W}(n)$, and thus
\[\mathbb E_nY_3(n+1) = p_{135}(\hat{{\bf W}}(n)) + p_{234}(\hat{{\bf W}}(n))
\leq \alpha\hat W_3(n),\]
almost surely
for all $n$ large enough,
by Lemmas~\ref{lem:W3} and~\ref{lem:p135p234}. 
Therefore, almost surely for $n$ large enough, $r_n \le -\delta Z_n$, 
for some constant $\delta>0$. 
As a consequence, almost surely there exists $m\ge 1$, such that for all $n>m$, 
\[Z_n \le  \gamma_{m,n} \cdot Z_m  +\sum_{i=m+1}^{n} \frac{\gamma_{i,n}\cdot \Delta M_3(i)}{i^\alpha},\]
where $\gamma_{i,n}: = \prod_{j=i+1}^n (1-\frac{\delta}j)$, for all $i\le n$ (with the convention that $\gamma_{n,n}=1$). 
Recall that by definition $|\Delta M_3(i)|\le 1$, almost surely for all $i\ge 1$. Thus, by Doob's $L^2$-inequality, one has as $m\to +\infty$,  
\[\mathbb P\left(\sup_{n\ge m} \left|\sum_{i=m+1}^{n}
 \frac{\gamma_{i,n} \cdot\Delta M_3(i)}{
  i^{\alpha}}\right| \ge  \frac{1}{m^{2\alpha-1-\frac 35}}\right) = \mathcal O\Big(\frac 1{m^{6/5}}\Big).\]
By Borel-Cantelli, we deduce that almost surely, one has for all $m$ large enough,   
\[ 
\sup_{n\ge m}\left|\sum_{i=m+1}^{n} \frac{\gamma_{i,n} \cdot\Delta M_3(i)}{ i^{\alpha}}\right| \le \frac{1}{m^{2\alpha-1-\frac 35}}. \]
The lemma follows, since $2\alpha-1-\frac 35>0$, and for any fixed $m\ge 1$, $\gamma_{m,n} \to 0$, as $n\to \infty$.
\end{proof}

%%%%%%%%%%%%%%%%%%%%%%%%%%%%%%%%%%%%%%%%%%%%%%%%%%%%%%%%%%%%%%%%%%%%%%

\subsection{Convergence of $\hat{{\bf W}}(n)$}\label{subsec:convW}
Our next goal is to prove the following proposition. 
\begin{proposition}\label{lem:conv}
Almost surely, there exists some (random) real $\chi \in [0,1]$, such that as $n\to \infty$, 
\[\frac{W_i(n)}{n} \to \chi ,\quad \forall i=1,2,\quad \text{and}\quad \frac{W_i(n)}{n} \to  1- \chi ,\quad \forall i=4,5.\]
\end{proposition}
We start with a computation giving the probability to reinforce edge~$2$, which is similar to Lemma~\ref{lem:prob_zigzag}. 
\begin{lemma}\label{lem:prob_zigzag.2}
One has for all $w\in \mathcal E$, 
\begin{linenomath}
\begin{align*}
p_{12}(w) + p_{234}(w)  =  \frac{w_2w_3 + w_1w_2w_5 + w_1w_2w_4}{w_3 + w_2w_5+w_1w_4}.    
\end{align*}
\end{linenomath}
\end{lemma}
\begin{proof}
Note that $p_{12}(w) + p_{234}(w)$ is equal to the probability that the last step before reaching the vertex $F$ is through  edge 2. Let us compute this probability by decomposing with respect to the first step, which is either through edge 1 (jumping on the left vertex), or through edge 4 (jumping on the right vertex), hence we will write
\begin{equation}\label{decompd}
p_{12}(w) + p_{234}(w)=p^{\ell}(w) + p^{r}(w).
\end{equation}
\begin{figure}
\begin{center}\includegraphics[width=12cm]{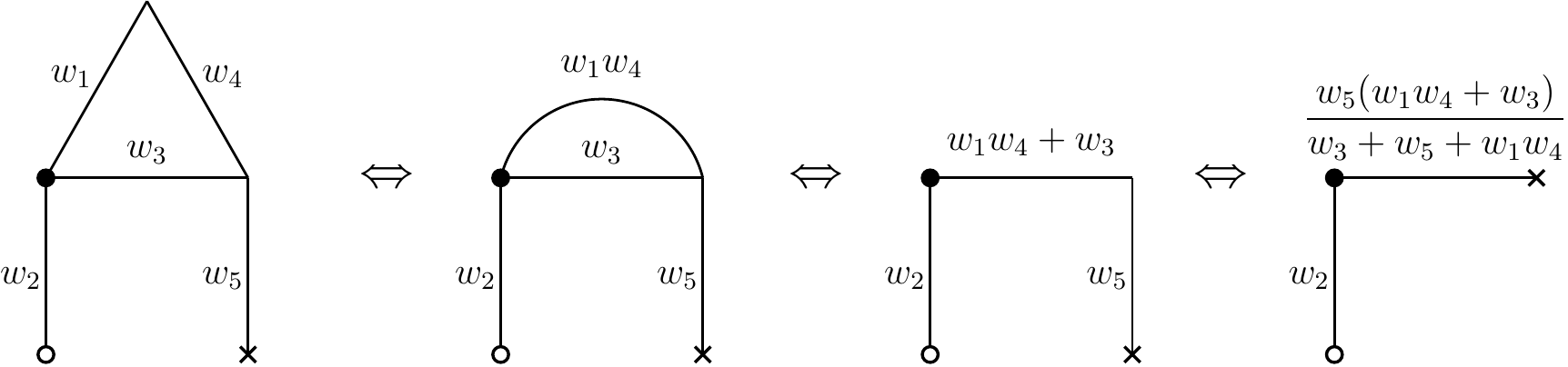}\end{center}
\caption{Calculation of $p^{\ell}(w)/w_1$ (i.e.\ the probability to reach the circled vertex before the crossed vertex starting from the black vertex) for $w\in\mathcal E$ (in particular, we use $w_1+w_4=1$).}
\label{fig:p2}
\end{figure}
For $w\in\mathcal E$, the probability to jump on the left vertex is $w_1$, and once on the left vertex, we need to compute the probability to cross edge 2 before crossing edge 5, which is easily done through graph transformations similar to those done in the proof of Lemma~\ref{lem:prob_zigzag}; see Figure~\ref{fig:p2}. One obtains:
\begin{linenomath}
\begin{align}
p^\ell(w)&=w_1\times \frac{w_2(w_3+w_5+w_1w_4)}{w_2(w_3+w_5+w_1w_4)+w_5(w_3+w_1w_4)}\\ \label{decompd1}
&=w_1\times \frac{w_2w_3+w_2w_5+w_1w_2w_4}{w_3+w_2w_5+w_1w_4},
\end{align}
\end{linenomath}
where we used that $w_2+w_5=1$.

Now, using symmetry, one has
\begin{linenomath}
\begin{align}
p^r(w)&=w_4\times\left(1- \frac{w_5w_3+w_2w_5+w_1w_5w_4}{w_3+w_2w_5+w_1w_4}\right)\\ \label{decompd2}
&=w_4\times \frac{w_2w_3+w_1w_2w_4}{w_3+w_2w_5+w_1w_4}.
\end{align}
\end{linenomath}
One can now easily conclude using \eqref{decompd}, by adding up \eqref{decompd1} with \eqref{decompd2} and using that $w_1+w_4=1$.
\end{proof}
We next deduce the following bound on $F_2(w)$ (the second coordinate of the function $F(w)$ from~\eqref{def.F}). 
\begin{lemma}\label{lem:F2}
For any $w \in \mathcal E$, we have
\[|F_2(w)| \le \frac{w_3}{2}.\]
\end{lemma}
\begin{proof}
By Lemma~\ref{lem:prob_zigzag.2}, for any $w\in \mathcal E$, 
\begin{equation}\label{eq.F2}
F_2(w) = p_{12}(w) +p_{234}(w)- w_2 = \frac{(w_1-w_2)w_2w_5}{w_3 + w_2w_5+w_1w_4}.
\end{equation}
Note now that since $w_1-w_2 = w_5-w_4$, either $w_1\ge w_2$, or $w_4\ge w_5$. 
In the first case, using also that $w_4 \ge w_5-w_3$, we deduce $w_3+w_1w_4\ge w_2w_5$. 
By symmetry, the same holds when $w_4\ge w_5$. 
We thus get
\[|F_2(w)|\le \frac{|w_1-w_2|}{2}\le \frac{w_3}{2},\quad \text{for all }w\in \mathcal E,\]
where we have used $|w_1-w_2|\le w_3$ in the second inequality.
\end{proof}

\begin{proof}[Proof of Proposition~\ref{lem:conv}]
Iterating Equation~\eqref{eq:algo_sto}, we get that, for all $n\geq 0$
\begin{equation}\label{hatW}
\hat{{\bf W}}(n) = \hat{{\bf W}}(0) + \sum_{i=0}^{n-1} \frac1{i+3}\big(F(\hat{{\bf W}}(i))+\Delta {\bs M}(i+1)\big). 
\end{equation}
where we recall that $\Delta \bs M(n+1):= Y(n+1)-\mathbb E_nY(n+1)$ with $Y(n+1) := {\bf W}(n+1) -{\bf W}(n)$, and where $F$ is defined in Equation~\eqref{def.F}.
By definition of the model, $\|Y(n+1)\|_1\leq 3$ almost surely, and thus $\|\Delta {\bs M}(i+1)\|_1\leq 3$ almost surely, which implies that the martingale 
\[\hat {\bs M}(n):=\sum_{i=0}^{n-1} \frac{\Delta {\bs M}(i+1)}{i+3}\]
is bounded in $L^2$ and thus converges almost surely when $n\to +\infty$.
By Lemma~\ref{lem:setE}, $\hat{{\bf W}}(n) \in \mathcal E$, for all $n\ge 0$.
Thus Lemma~\ref{lem:F2} gives $|F_2(\hat {{\bf W}}(n))|\leq \hat W_3(n)/2$, for all $n\ge 0$, which implies using Lemma~\ref{lem:W3poldec} that
\[\hat W_2(n) = \hat W_2(0) + \sum_{i=0}^{n-1}\frac{F_2(\hat{{\bf W}}(i))}{i+3} 
+ \sum_{i=0}^{n-1}\frac{\Delta M_2(i+1)}{i+3},\]
converges almost surely when $n\to+\infty$. The proposition follows, since by Lemma~\ref{lem:W3}, one has $\hat W_1(n) - \hat W_2(n)\to 0$, and by Lemma~\ref{lem:setE}, one has $\hat W_4(n) = 1- \hat W_1(n)$, and $\hat W_5(n) =1- \hat W_2(n)$, for all $n\ge 0$. 
\end{proof}

%%%%%%%%%%%%%%%%%%%%%%%%%%%%%%%%%%%%%%%%%%%%%%%%%%%%%%%%%%%%%%%

\subsection{On the absence of convergence to $0$ or $1$}\label{subsec:nondegenerate}
The last step of the proof is to exclude the convergence toward an extremal point, that is we prove the following proposition. 
\begin{proposition}\label{prop:nonconv} Almost surely, 
\[\lim_{n\to \infty} \frac{W_1(n)}{n} \notin \{0,1\}.\] 
\end{proposition}
Note that by symmetry it suffices to exclude the possibility of converging to $1$. We prove this by contradiction, 
and start with the following fact. 

\begin{lemma}\label{lem:W5poldec}
For all $\alpha\in(0,1)$, on the event where 
\[\lim_{n\to \infty} \frac{W_1(n)}{n} = 1,\quad \text{and} \quad \lim_{n\to +\infty} \frac{W_3(n)}{n^{\alpha}} = 0,\]
both hold, we have almost surely for any $\beta>\alpha$, 
\[\lim_{n\to +\infty} \frac{W_5(n)}{n^\beta} = 0.\]
\end{lemma}
\begin{proof}
Fix $\alpha \in (0,1)$ and 
assume that both $W_1(n)/n \to 1$ and $W_3(n) / n^\alpha \to 0$ when $n\to+\infty$. 
Assume by contradiction that there exists $\beta>\alpha$, such that  
$\limsup_{n\to+\infty} W_5(n)/n^\beta >0$. 
Without loss of generality one can even assume that $\limsup_{n\to+\infty} W_5(n)/n^\beta >1$, by taking a smaller $\beta$ if necessary. 
In other words, letting 
\[E:= \left\{\lim_{n\to \infty} \frac{W_3(n)}{n^\alpha} = 0, \ \text{ and }\lim_{n\to \infty} \frac{W_1(n)}{n} = 1\right\},\quad \text{and} \quad 
E':=E \cap \left\{\limsup_{n\to+\infty} \frac{W_5(n)}{n^\beta} >1\right\},\]
our aim is to show that $\P(E') = 0$.

For $m\ge 1$ integer, define 
\[E_m := \left\{W_3(n)\le n^\alpha, \ \text{and}\ W_2(n)\ge 3(n+2)/4  \quad \text{for all}\ n\ge m\right\}. \]
By definition, and using that $W_2(m) \ge W_1(m) - W_3(m)$, for all $m\ge 0$, one has {that $E\subset \cup_mE_m$, and therefore}
\[\lim_{m\to \infty} \P(E\cap E_m^c) = 0.\]
Thus it amounts to show that 
\[
\lim_{m\to \infty} \P(E_m \cap E') = 0. 
\]
Note now that by conditioning with respect to the first time $n\ge m$ when $W_5(n) \ge n^\beta$, it suffices in fact to show that almost surely 
\begin{equation}\label{goal:coupling}
\lim_{m\to \infty} \P(E_m \cap E \mid \mathcal F_m) \cdot \mathbf{1}\{W_5(m) \ge m^\beta\} = 0, 
\end{equation}
where $\mathcal F_m = \sigma({\bf W}(0),\dots,{\bf W}(m))$. Thus the rest of the proof consists in proving \eqref{goal:coupling}. 
The idea is to show that for any integer $m\ge 1$, on the event that $\{W_5(m)\ge m^\beta\}$, 
the process $(W_2(n))_{n\ge m}$ can be coupled with another process $(R_n)_{n\ge m}$, 
in a way that outside an event with vanishing probability as $m\to \infty$, one has  
$W_2(n) \le R_n$ for all $n\ge m$, and $\limsup_{n\to \infty} R_n/n <1$, from which \eqref{goal:coupling} follows.

We proceed with the details now. Fix $\gamma\in (0,1)$, such that $1+\alpha<\beta + \gamma$. 
Let $m\ge 1$ be given, and conditionally on $\mathcal F_m$, we define the process $(R_n)_{n\geq m}$ as follows:
$R_m = W_2(m)$, and for all $n\geq m$,
\begin{equation}\label{eq:qn}
q_n:=\P(R_{n+1} = R_n+1\mid \G_n) 
= 1- \P(R_{n+1} = R_n\mid \G_n) 
= \frac{R_n + R_n^\gamma}{n+R_n^{\gamma}},
\end{equation}
where $\G_n=\mathcal F_m\vee \sigma(R_m,\dots,R_n)$.

$\bullet$ First, we prove that, for all $m\geq 1$, if
we set
\begin{equation}\label{eq:defA}
\mathcal A_m:=
\bigg\{\inf_{n\ge m} \frac{R_n}{n} >\frac {3}{5}\bigg\} \cap 
\Big\{\inf_{n\ge m} \frac{n-R_n}n >  \frac3{5m^{1-\beta}}\Big\},
\end{equation}
then almost surely on the event $\{W_5(m) \ge m^\beta\}\cap \{W_2(m) \ge 3(m+2)/4\}$, one has 
\begin{equation}\label{RnRubin}
\mathbb P(\mathcal A_m^c\mid \mathcal F_m)  = \mathcal O(m^{-\delta}),
\end{equation}
where the implicit constant in the $\mathcal O$ is deterministic, and $\delta = \delta(\beta,\gamma)$ is some positive constant depending only on $\beta$ and $\gamma$.
\begin{figure}
\begin{center}
\includegraphics[width=14cm]{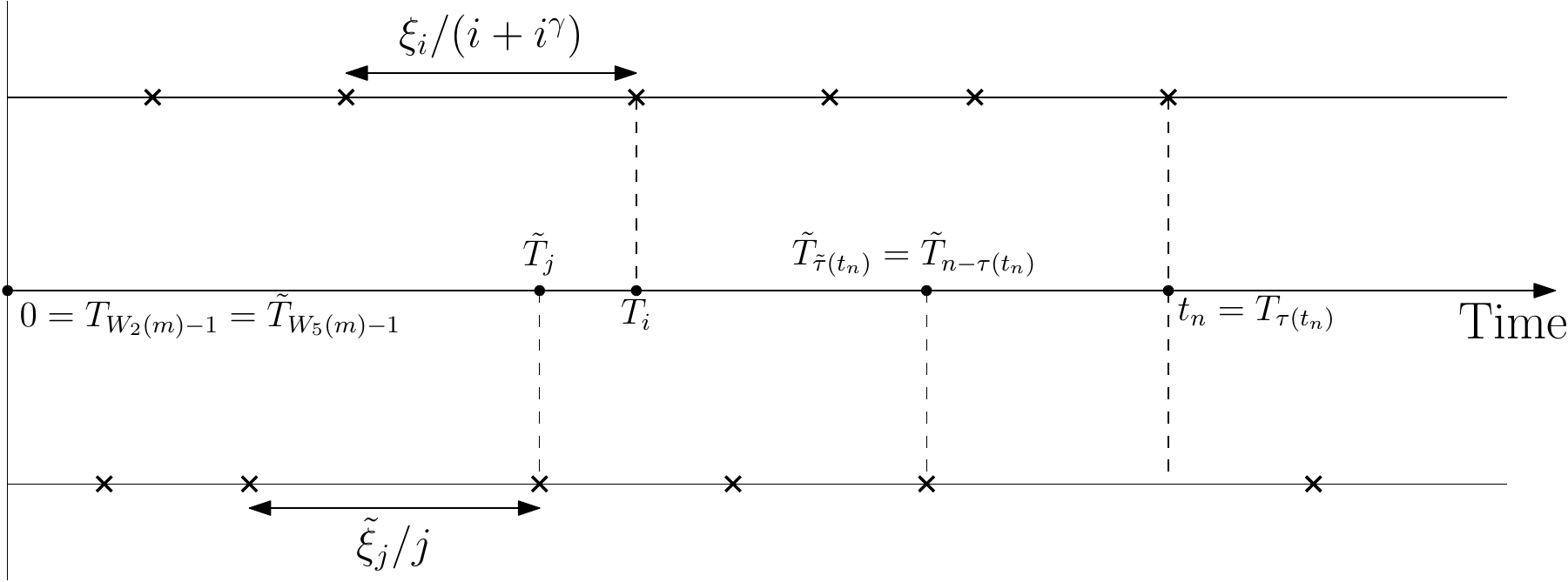}
\end{center}
\caption{Rubin's construction for the proof of Lemma~\ref{lem:W5poldec}.}
\label{fig:rubin.2}
\end{figure}
To do this, we use again Rubin's construction; see Figure~\ref{fig:rubin.2}:
Let $(\xi_i)_{i\ge 1}$ and $(\tilde \xi_i)_{i\ge 1}$ 
be two independent sequences of independent exponential random variables with mean~$1$ 
(also independent of the process $(W_2(n))_{n\ge 1})$. 
For all $m\ge 0$ and  $i\ge 0$, set
\[T_i := \sum_{j=W_2(m)}^i \frac{\xi_j}{j + j^\gamma}, \quad \text{and} \quad 
\tilde T_i := \sum_{j=W_5(m)}^i \frac{\tilde \xi_j}{j},\]
with the convention that $T_i =0$ for $i< W_2(m)$, and  $\tilde T_i = 0$, for $i< W_5(m)$.
For all  $t\ge 0$, set
\[\tau(t) := \sup \{i\ge 0 : T_i\le t\},\quad \text{and}\quad \tilde 
\tau(t) := \sup \{i\ge 0 : \tilde T_i\le t\},\]
and for all $n\ge m$, 
\[t_n:= \inf\{t\ge m : \tau(t) + \tilde \tau(t) \ge n\}.\]
Standard properties of independent exponential random variables imply 
that $(\tau(t_n))_{n\ge m}$ and $(R_n)_{n\ge m}$ have the same law.
Note that for all $m\geq 1$, and $i\ge W_2(m)$,
\begin{equation}\label{eq:Ti}
T_i  = M_i + \log \Big(\frac{i}{W_2(m)}\Big) 
+ \mathcal O\Big(\frac 1{W_2(m)^{1-\gamma}}\Big), \quad \text{with}\quad  M_i :=\sum_{j=W_2(m)}^i \frac{\xi_j-1}{j + j^\gamma},
\end{equation}
when $m\to+\infty$,
and for all $i\ge W_5(m)$,
\begin{equation}\label{eq:tildeTi}
\tilde T_i  = \tilde M_i + \log \Big(\frac{i}{W_5(m)}\Big) + \mathcal O\Big(\frac 1{W_5(m)}\Big), 
\quad \text{with}\quad  \tilde M_i :=\sum_{j=W_5(m)}^i \frac{\tilde \xi_j-1}{j},
\end{equation}
when $m\to+\infty$.
By Doob's $L^2$-maximal inequality, we get that, almost surely 
\begin{equation}\label{eq:mart1}
\mathbb P\bigg(\sup_{i\ge W_2(m)} |M_i| > \frac 1{W_2(m)^{\nicefrac14}}\ \Big|\ W_2(m) \bigg) 
\le 4W_2(m)^{\nicefrac12} \sum_{i\ge W_2(m)} \frac 1{i^2} 
= \mathcal O\Big(\frac 1{W_2(m)^{\nicefrac12}}\Big),
\end{equation}
when $m\to+\infty$,
and similarly,
\begin{equation}\label{eq:mart2}
\mathbb P\bigg(\sup_{i\ge W_5(m)} |\tilde M_i| 
> \frac 1{W_5(m)^{\nicefrac14}}\ \Big|\ W_5(m)\bigg) 
= \mathcal O\Big(\frac 1{W_5(m)^{\nicefrac12}}\Big),
\end{equation}
when $m\to+\infty$.
Moreover, by definition 
\begin{equation}\label{eq:TT}
T_{\tau(t_n)} = \tilde T_{n- \tau(t_n)}  + \mathcal O(\Gamma_m + \tilde \Gamma_m),
\end{equation}
with 
\[\Gamma_m = \sup_{j\ge W_2(m)} \xi_j/j, \quad \text{and}\quad 
\tilde \Gamma_m = \sup_{j\ge W_5(m) }\tilde \xi_j / j.\]
Note that, for all $m$ large enough, 
\begin{equation}\label{eq:gammas}
\mathbb P\Big(\Gamma_m>\frac 1{W_2(m)^{\nicefrac12}}\mid W_2(m) \Big)\le 
\exp\Big(-\frac{\sqrt{W_2(m)}}2\Big) 
\text{ and }
\mathbb P\Big(\tilde \Gamma_m>\frac 1{W_5(m)^{\nicefrac12}}\mid W_5(m)\Big)
\le \exp\Big(-\frac{\sqrt{W_5(m)}}2\Big).
\end{equation}
Taking the exponential in Equation~\eqref{eq:TT} gives
\begin{equation}\label{eq:exp_eps}
\frac{\tau(t_n)}{n-\tau(t_n)} \cdot \frac{W_5(m)}{W_2(m)}
=\exp(\varepsilon_n),
\end{equation}
where, by Equations~\eqref{eq:Ti}, \eqref{eq:tildeTi}, \eqref{eq:mart1}, \eqref{eq:mart2} and~\eqref{eq:gammas}, there exists $\delta=\delta(\beta,\gamma)>0$ such that almost surely on the event $\{W_5(m) \ge m^\beta\} \cap \{W_2(m) \ge 3(m+2)/4\}$, 
\[\mathbb P(\sup_{n\ge m} |\varepsilon_n|> m^{-\delta} \mid \mathcal F_m) = \mathcal O(m^{-\delta}).\]
Since, by Lemma~\ref{lem:setE}, $W_5(m) = m+2-W_2(m)$, we get that on the event $\{W_2(m)\geq 3(m+2)/4\}$, one has 
$W_5(m)/W_2(m)\leq \nicefrac13$, and thus, by Equation~\eqref{eq:exp_eps},
\begin{equation}\label{eq:borneinf}
\tau(t_n)\geq (n-\tau(t_n)) 3\mathrm e^{\varepsilon_n}
\quad\Longrightarrow\quad 
\tau(t_n) \geq \frac{3\mathrm e^{\varepsilon_n} n}{1+3\mathrm e^{\varepsilon_n}}
= \frac{3n(1-\mathcal O(m^{-\delta}))}{4},
\end{equation}
where the last equality holds on an event of probability at least $1-\mathcal O(m^{-\delta})$ when $m\to+\infty$.
Similarly, on the event $\{W_5(m)\geq m^{\beta}\}$, we have $W_5(m)/W_2(m)\geq m^{\beta-1}$, and thus, by Equation~\eqref{eq:exp_eps},
\[(n-\tau(t_n)) m^{1-\beta} \mathrm e^{\varepsilon_n}\geq \tau(t_n)
\geq  \frac{3n(1-\mathcal O(m^{-\delta}))}{4},\]
where the last inequality comes from Equation~\eqref{eq:borneinf}. 
This implies
\[n-\tau(t_n)\geq \frac{3n}{4m^{1-\beta}} (1-\mathcal O(m^{-\delta})),\]
when $m\to+\infty$, on an event of probability at least $1-\mathcal O(m^{-\delta})$.
Since $(R_n)_{n\ge m}$ and $(\tau(t_n))_{n\ge m}$ have the same law by construction, this concludes the proof of~\eqref{RnRubin}.

$\bullet$ To conclude we just need to show that 
there exists a coupling of $(W_2(n))_{n\geq m}$ and $(R_n)_{n\geq m}$, 
such that almost surely on the event $\mathcal A_m\cap E_m$ (see~Equation~\eqref{eq:defA} for the definition of $\mathcal A_m$), 
one has $W_2(n)\leq R_n$ for all $n\geq m$, at least for $m$ large enough. Indeed, this would 
prove that on $\mathcal A_m\cap E_m$, the sequence $(W_1(n)/n)_{n\ge m}$ cannot converge to $1$, or otherwise stated that for all $m$ large enough, almost surely,  
\[\P(\mathcal A_m \cap E_m \cap E\mid \mathcal F_m)=0.\] 
Together with \eqref{RnRubin}, this would conclude the proof of~\eqref{goal:coupling}. So let us prove the existence of the desired coupling now.

Recall that, by Lemma~\ref{lem:F2}, for all $n\ge m$, on the event $\{W_3(n) \le n^\alpha\}$, 
\[p_n:=\P(W_2(n+1) = W_2(n) + 1\mid \mathcal F_n)\le \frac{W_2(n) + n^\alpha}{n+2}. \]
To show that our coupling exists, it is enough to prove that, for all $n\geq m$, 
if $W_2(n)\leq R_n$, then $p_n\leq q_n$, where $q_n$ is defined in Equation~\eqref{eq:qn}. {Indeed, if $W_2(n)\leq R_n$ and $p_n\leq q_n$, then there exists a one-step coupling such that $W_2(n+1)\leq R_{n+1}$, and we can proceed by induction.}
Note that $q_n\geq p_n$ is implied by
\[(n+2) (R_n + R_n^\gamma) \ge  (n+R_n^\gamma)(W_2(n) + n^{\alpha}).\]
{Developing and u}sing the induction hypothesis (i.e.\ $W_2(n) \le R_n$), 
it suffices to show that
\[R_n^\gamma(n-R_n-n^\alpha) \ge n^{1+\alpha},\]
which is indeed true on $\mathcal A_m$, since on this event  
\[R_n^\gamma (n-R_n - n^\alpha) \ge \frac 35 n^\gamma\Big(\frac 35 n^\beta - n^\alpha\Big),\] 
which is well larger than $n^{1+\alpha}$, for all $n$ large enough, since by hypothesis {$\gamma\in(0,1)$ and} $\gamma + \beta > 1+ \alpha$.
This concludes the proof of~\eqref{goal:coupling}, and of the lemma. 
\end{proof}

\begin{lemma}\label{lem:seedW5}
Let $\alpha \in (0,1)$ be given. 
On the event
\[\mathcal A(\alpha) = \big\{W_i(n) = \mathcal O(n^{\alpha}), \text{ for }i=3,5\big\}
\]
one has almost surely for any $\beta>\max(2\alpha -1,\frac \alpha 2)$, 
\[W_4(n) + W_5(n) \le W_3(n)+\mathcal O(n^{\beta} ).\]
\end{lemma}
For the proof of this lemma we need some intermediate results. The first one gives 
a lower bound on $F_4 + F_5-F_3$.  

\begin{lemma}\label{lem:F4F5F3}
For all $w\in \mathcal E$, such that $w_5\le \nicefrac 12$, we have 
\[F_4(w) + F_5(w) - F_3(w) \ge -8(w_3^2+w_5^2). \]
\end{lemma}
\begin{proof}
Note that 
\[F_4(w) + F_5(w) - F_3(w) = 2F_5(w) - 2p_{135}(w) + w_5 - w_4+w_3.\]
Recall that by Equation~\eqref{eq.F2} (using the symmetry of the model), we have 
\[F_5(w) = \frac{(w_4-w_5)w_2w_5}{w_3+w_1w_4+w_2w_5}.\]
Also recall that (see Equation~\eqref{eq.p135}),
\begin{linenomath}\begin{align*}
p_{135}(w) 
&= \frac{w_1w_3w_5}{w_3(1+w_4+w_1w_4) + w_2w_4 + w_2w_5 + w_1w_4^2 + w_1w_4w_5 }\\
&\leq \frac{w_1w_3w_5}{w_3+ (w_3+w_2)w_4 + w_2w_5}
\leq  \frac{w_1w_3w_5}{w_3+ w_1w_4 + w_2w_5},
\end{align*}\end{linenomath}
where we have used in the last inequality that $w_1\leq w_2+w_3$ for all $w\in\mathcal E$.
Using again that $w_1\le w_2+w_3$ for all $w\in\mathcal E$, 
and the fact that $w_3+w_1w_4+w_2w_5\geq w_5(1-w_5)\ge w_5/2$, 
for all $w\in\mathcal E$ such that $w_5\le \nicefrac 12$, we get that
\[p_{135}(w) \leq \frac{w_2w_3w_5}{w_3+w_1w_4 + w_2w_5} + 2w_3^2.\]
Therefore, 
\[F_5(w) - p_{135}(w) 
\ge -\frac{w_5w_2(w_3-w_4+w_5)}{w_3+w_1w_4+w_2w_5} - 2w_3^2,\]
and thus
\begin{linenomath}\begin{align}
F_4(w)&+F_5(w) - F_3(w)
\ge \frac{(w_3+w_1w_4-w_2w_5)(w_3-w_4+w_5)}{w_3+w_1w_4+w_2w_5} - 4w_3^2\notag\\
&\geq \frac{(w_3+w_4-w_5)(w_3-w_4+w_5)}{w_3+w_1w_4+w_2w_5} +
\frac{(w_5(1-w_2)-w_4(1-w_1))(w_3-w_4+w_5)}{w_3+w_1w_4+w_2w_5}- 4w_3^2,\label{eq:w4^2}
\end{align}\end{linenomath}
Recall that, for all $w\in\mathcal E$, $w_1+w_4 = w_2+w_5 = 1$, and thus $w_5(1-w_2)-w_4(1-w_1)=w_5^2-w_4^2\geq -w_4^2$.  
As a consequence, for all $w_5\le \nicefrac 12$,  
\[\frac{(w_5(1-w_2)-w_4(1-w_1))(w_3-w_4+w_5)}{w_3+w_1w_4+w_2w_5}
\geq -\frac{w_4^2(w_3-w_4+w_5)}{w_3+w_1w_4+w_2w_5}
\geq -\frac{w_4^2(w_3+w_5)}{w_3+w_1w_4+w_5-w_5^2}
\geq -2w_4^2,\]
where we used that, as $w_5\le\nicefrac 12$, $w_3+w_1w_4+w_5-w_5^2\ge w_3+w_5(1-w_5)\ge (w_3+w_5)/2$.
Then from Equation~\eqref{eq:w4^2}, we get
\begin{linenomath}\begin{align*}
F_4(w)+F_5(w) - F_3(w)
& \ge  \frac{(w_3+w_4-w_5)(w_3-w_4+w_5)}{w_3+w_1w_4+w_2w_5} - 4w_3^2 - 2w_4^2\\
& \ge \frac{w_3^2 - (w_4-w_5)^2}{w_3+w_1w_4+w_2w_5} - 4w_3^2 - 2w_4^2
\geq - 4w_3^2 - 2(w_3+w_5)^2, 
\end{align*}\end{linenomath}
using that $|w_4-w_5|\le w_3$, for all $w\in\mathcal E$. 
This concludes the proof because $(w_3+w_5)^2\leq 2 w_3^2+2w_5^2$.
\end{proof}

The second result we shall need is the following general fact, which will be used at several places during the rest of the proof. 
For a process $(M_n)_{n\ge 0}$, we write $\Delta M_n : = M_{n+1}- M_n$, for all $n\ge 0$. 
\begin{lemma}\label{lem:martingale}
Let $a, b, c\in (0,1)$, be such that $b<a$ and $1<2a+c$.\\
{\bf (i)} Let $(A_n)_{n\geq 1}$ be a sequence of real random variables. On the event $\{A_n=\mathcal O(n^{b-1})\}$, we have almost surely as $m\to+\infty$,
\[\sup_{n\ge m} \sum_{i=m}^n \frac{A_i}{(i+3)^a}=\mathcal O(m^{b-a}).\]
{\bf (ii)} Let $(M_n)_{n\leq 0}$ be a real martingale such that $|\Delta M_n|\le 1$, almost surely for all $n\ge 0$.
On the event $\{\mathbb E_n [(\Delta M_n)^2] = \mathcal O(n^{-c})\}$, we have almost surely when $m\to+\infty$,
\[\sup_{n\ge m }\left| \sum_{i=m}^n \frac{\Delta M_i}{(i+3)^a} \right| = \mathcal O(m^{\kappa-a}),\]
for all $\kappa\in (\frac{1-c}2,a)$.
\end{lemma}

\begin{proof}
{\bf (i)} is straighforward. For {\bf (ii)}, 
we fix $\frac{1-c}2<\kappa<a$, 
and then $\varepsilon>0$, such that $\frac{1-c+\varepsilon}{2}<\kappa$, and $\hat\kappa:=\kappa+\nicefrac\varepsilon2<a$.
For $m\le n$, define the event 
\[\mathcal A_{m,n}:= \{ \mathbb  E_i[(\Delta M_i)^2] \le i^{-c+\varepsilon} \quad \forall m\le i \le n\}.\]
We have, for all $n\ge m\geq 1$,
\begin{linenomath}\begin{align*}
\mathbb P\bigg(\sum_{i=m}^n \frac{\Delta M_i}{(i+3)^a}
\geq m^{\hat \kappa-a},\, \mathcal A_{m,n}\bigg)
\leq \exp\big(-m^{\nicefrac\varepsilon2}\big)
\mathbb E\bigg[\prod_{i=m}^n \exp\bigg(m^{a- \kappa}
\frac{\Delta M_i}{(i+3)^a}\bigg){\bf 1}_{\mathcal A_{m,n}}\bigg].
\end{align*}\end{linenomath}
Using the bound $|\Delta M_n|\le 1$, and a Taylor expansion, we get for all $n\geq m$, on the event $\mathcal A_{m,n}$, 
\[\mathbb E_n\bigg[\exp\bigg(m^{a-\kappa}
\frac{\Delta M_n}{(n+3)^a}\bigg)\bigg]
= 1+\mathcal O\Big(m^{2a-2\kappa}
\frac{\mathbb E_n[(\Delta M_n)^2]}{(n+3)^{2a}}\Big)
\leq 1+\mathcal O\big(n^{-2\kappa-c+\varepsilon}\big),
\]
where the constant in the $\mathcal O$-term is deterministic.
By induction, and since $2\kappa+c-\varepsilon>1$,
we get that for all $m$ large enough,
\[\mathbb E\bigg[\prod_{i=m}^n \exp\bigg(m^{a-\kappa}
\frac{\Delta M_i}{(i+3)^a}\bigg){\bf 1}_{\mathcal A_{m,n}}\bigg]\leq 2,
\]
and thus for all $1\le m \le n$, with $m$ large enough, 
\[\mathbb P\bigg(\sum_{i=m}^n \frac{\Delta M_i}{(i+3)^a}
\geq m^{\hat \kappa-a}, \ \mathcal A_{m,n}\bigg)
\leq 2\exp\big(-m^{\nicefrac\varepsilon2}\big).\]
By symmetry and a union bound, we deduce that for all $m$ large enough,  
\[\mathbb P\bigg(\sup_{m\le n \le 2m} \left|\sum_{i=m}^n  \frac{\Delta M_i}{(i+3)^a}\right| 
\geq m^{\hat \kappa-a}, \ \mathcal A_{m,2m}\bigg)
\leq 4m\exp\big(-m^{\nicefrac\varepsilon2}\big).\]
Next, another union bound gives, for all $m$ large enough, 
\[\mathbb P\bigg(\sup_{n \ge m} \left|\sum_{i=m}^n  \frac{\Delta M_i}{(i+3)^a}\right| 
\geq R \cdot m^{\hat \kappa-a}, \ \mathcal A_{m,\infty}\bigg)
\leq \exp\big(-\frac 12 \cdot m^{\nicefrac\varepsilon2}\big),\]
with $R:=\sum_{i\ge 0} 2^{(\hat \kappa-a)i}$, which is finite since $\hat \kappa <a$. 
Then the result follows from Borel-Cantelli's lemma, since on the event  
$\{\mathbb E_n [(\Delta M_n)^2] = \mathcal O(n^{-c})\}$, almost surely $\mathcal A_{m,\infty}$ holds for all $m$ large enough. 
\end{proof}

We now prove Lemma~\ref{lem:seedW5}. 
\begin{proof}[Proof of Lemma~\ref{lem:seedW5}]
Consider the process $U(n)  = W_5(n) +W_4(n) -  W_3(n)$, and set $\hat U(n) := \frac{U(n)}{n+2}$. One has for any integers $m<n$, 
\begin{equation}\label{eq.Un}
\hat U(n) = \hat U(m) + \sum_{i=m}^{n-1} \frac{G(\hat{{\bf W}}(i))}{i+3} 
+ \sum_{i=m}^{n-1} \frac{\Delta \Phi(i)}{i+3},
\end{equation}
where $G(w)= F_5(w) +F_4(w)-F_3(w)$, 
and $\Delta \Phi(i) = Y(i+1) -\mathbb E_i Y(i+1)$, 
with $Y(i+1) = U(i+1)-U(i)$ for all $i\geq 0$.

Note first that, $|Y(n+1)|\le 2$ almost surely for all $n\ge 0$, by definition of the model, and thus also  $|\Delta \Phi(n)|\leq 4$.   
Note furthermore, that on $\mathcal A(\alpha)$, 
one has $W_4(n) = \mathcal O(n^\alpha)$, since $W_4(n) \le W_5(n) + W_3(n)$ (recall Lemma~\ref{lem:setE}), and thus $|\hat U(n)| = \mathcal O(n^{\alpha-1})$. Moreover, using Lemmas~\ref{lem:W3} and~\ref{lem:F2} (and the fact that if at some time $n$, $W_4$ increases by one unit, 
then either $W_3$ or $W_5$ also), 
we deduce that on $\mathcal A(\alpha)$, 
\[\mathbb E_n[|\Delta \Phi(n)|^2] \le \mathbb E_n[Y(n+1)^2] \le 4\cdot \mathbb P_n(Y(n+1)\neq 0)
=\mathcal O(\hat W_5(n)  + \hat W_3(n)) = \mathcal O(n^{\alpha-1}).\]
On the other hand, by Lemma~\ref{lem:F4F5F3}, on $\mathcal A(\alpha)$,
we have almost surely $G(\hat{{\bf W}}(i))  \ge  - \mathcal O(i^{2\alpha-2})$. 
Thus Lemma~\ref{lem:martingale} (applied with $a=1$, $b=2\alpha - 1$, and $c=1-\alpha$) and Equation~\eqref{eq.Un} (with $n$ taken large enough) give $\hat U(m) \le \mathcal O(m^{\beta-1})$, for any $\beta > \max(2\alpha - 1,\nicefrac \alpha 2)$, which proves the desired result.
\end{proof}

We deduce the following fact (recall the definition of the events $\mathcal A(\alpha)$ from Lemma~\ref{lem:seedW5}). 

\begin{lemma}\label{lem:polWi}
Let $\alpha \in (0,1)$ be given. One has almost surely, 
\[\mathcal A(\alpha)\subseteq \mathcal A(\beta),\] 
for any $\beta>\max(2\alpha-1,\frac \alpha 2)$.  
\end{lemma}

For the proof of this result, we will need some intermediate result. 
\begin{lemma}\label{lem:F5-F4}
\begin{itemize}
\item[{\bf (a)}] For all $c\in (\nicefrac12,\nicefrac34)$, there exist positive constants $\varepsilon$ and $C$, such that for all $r\in [0,1)$, and all $w\in \mathcal E$, with $w_3\le \varepsilon$ and 
$w_4+w_5 \le w_3+ r$,  one has
\[F_5(w) - F_4(w) \ge   c(w_4-w_5)- Cr. \]
\item[{\bf (b)}]
There exist positive constants $\varepsilon$ and $C$, such that for any $w\in \mathcal E$, with $w_3\le \varepsilon$, $w_4+w_5\le w_3+r$, and $w_4\le w_5$, one has 
\[\frac 92 F_4(w) - F_3(w) \ge -Cr.\]
\item[{\bf (c)}]  Let $\rho \in (0,\nicefrac14)$ be given. There exist positive constants $\varepsilon$ and $C$, 
such that for any $r\in [0,1)$, and any $w\in \mathcal E$, with $w_3\le \varepsilon$, $w_4+w_5\le w_3 + r$, and
$ w_4\le \rho w_3+r$, one has 
\[F_4(w) \ge -Cr.\]
\end{itemize}
\end{lemma}
\begin{proof}
Let us start with Part {\bf (a)}. 
Note that, if, under the assumption of the lemma, we have $w_4+w_5-w_3\geq 2w_3$, then $3w_3\leq w_4+w_5\leq w_3+r$, which implies $w_3\leq \nicefrac r2$, and thus $w_4+w_5\leq \nicefrac{3r}2$. In particular, we have that $w_3, w_4, w_5\in [0,2r)$.
Recall that, 
$F_5(w)-F_4(w) =  p_{135}(w) - p_{234}(w)+w_4-w_5\geq -p_{234}(w)-w_5$, 
and, by Equation~\eqref{eq.p135} (using the symmetry of the model), we have that, for all $w\in\mathcal E$,
$p_{234}(w)\leq w_4$.
Thus $F_5(w)-F_4(w)\geq -w_4-w_5\geq -4r$, which, using that $w_4\le 2r$, concludes the proof of {\bf (a)} in the case when $w_4+w_5-w_3\geq 2w_3$.
We now assume that $w_4+w_5-w_3< 2w_3$. This implies
\begin{equation}\label{eq:1/2w3}
\frac 1{2w_3}\ge \frac 1{w_3+w_4+w_5} \ge \frac1{2w_3(1+\frac {w_4+w_5- w_3}{2w_3})} \ge \frac{1}{2w_3} -\frac{r}{4w_3^2},
\end{equation}
using that $w_4+w_5\geq w_3$, for all $w\in\mathcal E$, for the first inequality.  
Using again~\eqref{eq.p135}, 
we get that, when $w_3, w_4, w_5\to 0$, with $(w_4+w_5)/3\le w_3\le w_4+w_5$,
\begin{linenomath}\begin{align*}
F_5(w) - F_4(w) & = p_{135}(w) - p_{234}(w)+w_4-w_5 \\
& = w_4-w_5 + \frac{w_3w_5}{w_3+w_4+w_5}(1-o(1)) - \frac{w_3w_4}{2(w_3+w_4+w_5)}(1+o(1))\\
& \ge  w_4-w_5 + \frac{w_5}{2}(1-o(1))  - \frac{w_4}{4}(1+o(1)) - \frac{r w_5(1+o(1))}{4w_3}\\
&\ge \frac{3w_4}{4}(1-o(1)) - \frac{w_5}{2}(1+o(1)) - \frac{3r(1+o(1))}{4},
\end{align*}\end{linenomath}
because $w_4+w_5-w_3<2w_3$ implies $w_5\leq 3w_3$.
This concludes the proof of {\bf (a)}.

\bigskip
We prove now Part {\bf (b)}. First note that if $w_4+w_5-w_3\ge 2w_3$, then we have as in Part {\bf (a)} that $w_3,w_4,w_5\in [0,2r]$, and since $F_3(w) \le 0$ by Lemma~\ref{lem:W3}, we deduce that $\nicefrac 92 \cdot F_4(w)-F_3(w) \ge -9w_4/2\ge -9r$, proving the result.
So we may assume now that $w_4+w_5-w_3< 2w_3$.
In this case
\[\frac 92 F_4(w) - F_3(w) = \frac 92 \big(F_5(w) +w_5-w_4 + p_{234}(w) - p_{135}(w)\big) - F_3(w). \]
Using Equation~\eqref{eq.F2}, we have, when $w_3, w_4, w_5\to 0$,
\[
F_5(w) +w_5-w_4
= \frac{(w_5-w_4)(w_3+w_1w_4)}{w_3+w_1w_4+w_2w_5}
= \frac{(w_5-w_4)(w_3+w_4)(1+o(1))}{w_3+w_4+w_5}. 
\]
Using Equation~\eqref{eq:1/2w3}, and the fact that $w_4\leq w_5$, we get
\[F_5(w) +w_5-w_4
\geq \frac{(w_5-w_4)(1+o(1))}{2}\Big(1-\frac{r}{2w_3}\Big)
\geq \frac{(w_5-w_4)(1+o(1))}{2}- \frac r4(1+o(1)),
\]
using that $w_5-w_4 \le w_3$.
In the proof of {\bf (a)}, we have shown that
\[p_{234}-p_{135}
=\frac{w_3w_4(1+o(1))}{2(w_3+w_4+w_5)} -\frac{w_3w_5(1+o(1))}{w_3+w_4+w_5}\geq \frac{(w_4-2w_5)(1+o(1))}{4}.\]
Using in addition that by assumption $w_3\geq w_4+w_5-r$, we get
\begin{linenomath}\begin{align}\label{steak}
F_3(w)
&=p_{135}+p_{234}-w_3 
= \frac{w_3w_5(1+o(1))}{w_3+w_4+w_5}
+\frac{w_3w_4(1+o(1))}{2(w_3+w_4+w_5)}-w_3\\ \nonumber
&\leq -\frac{w_5(1+o(1))}2 
-\frac{3w_4(1+o(1))}4+r.
\end{align}\end{linenomath}
In total, we thus get
\begin{linenomath}\begin{align*}
\frac 92 F_4(w) - F_3(w)
& \ge -\frac{9w_4(1+o(1))}8- 
\frac{9 r (1+o(1))}{8}
+ \Big(\frac{w_5}{2} + \frac{3w_4}{4}-r\Big)(1+o(1)) \\
& \ge -\frac{3w_4}8(1+o(1))   + \frac{w_5}{2}(1-o(1))-
 \frac{17 r (1+o(1))}{8},\\
&\geq \frac{w_5(1+o(1))}8- 
 \frac{17 r (1+o(1))}{8},
\end{align*}\end{linenomath}
because $w_4\leq w_5$,
which concludes the proof of {\bf (b)}.
\bigskip 

Finally we prove {\bf (c)}. Assuming again that $w_4+w_5-w_3<2w_3$ (as otherwise we conclude as in Part {\bf (b)}), 
we get when $w_3\to 0$
(and as consequence $w_4,w_5\to 0$ also), 
\begin{linenomath}\begin{align*}
F_4(w) & = F_5(w) + p_{234}(w) - p_{135}(w) + w_5- w_4 \\
& = \frac{(w_5-w_4)(w_3+w_1w_4)}{w_3+w_1w_4+w_2w_5} + p_{234}(w) - p_{135}(w) \\
 & \ge \frac{(w_5-w_4)(w_3+w_1w_4)}{w_3+w_1w_4+w_2w_5} + \frac{w_3w_4(1-o(1))}{2(w_3+w_4+w_5)} - \frac{w_3w_5}{w_3+w_1w_4+w_2w_5}\\
 & \ge -w_4(1+o(1))\cdot \frac{w_3+w_1w_4-w_1w_5}{w_3+w_4+w_5} +  \frac{w_3w_4(1-o(1))}{2(w_3+w_4+w_5)}\\
 & = \frac{w_4\left[w_5(1-o(1))-(\frac{w_3}{2}+w_4)(1+o(1))\right]}{w_3+w_4+w_5}.
 \end{align*}\end{linenomath}
 Using the fact that, for all $w\in\mathcal E$, $w_5\geq w_3-w_4$, and the fact that, by assumption, $w_4\leq \rho w_3+r$, we get
 \[w_5-\frac{w_3}2-w_4\geq \frac{w_3}2-2w_4\geq \frac{w_3}2-2\rho w_3-2r=
 \frac{w_3}2(1-4\rho)-2r,\]
 which implies 
 \[F_4(w)\ge \frac{w_4w_3(1 -4\rho)}{2(w_3+w_4+w_5)}(1-o(1)) - {2{r(1+o(1))}}
\geq -{2r(1+o(1))}, 
 \]
 since $\rho<\nicefrac14$ by assumption.
\end{proof}

We are now ready to prove Lemma~\ref{lem:polWi}.
\begin{proof}[Proof of Lemma~\ref{lem:polWi}]
We fix $\alpha\in(0,1)$.
Recall that 
\[\mathcal A(\alpha) = \big\{W_3(n)=\mathcal O(n^\alpha) \text{ and }W_5(n)=\mathcal O(n^\alpha) \big\}.\]
The proof is divided in three steps. 

\bigskip

\underline{First step.} Fix $c\in ({\nicefrac34},1)$. 
Consider the process $U(n) = \frac{W_5(n) - W_4(n)}{(n+2)^{1-c}}$.  By Equation~\eqref{eq:algo_sto}, we have for $n\ge 1$, 
\begin{equation}\label{eq:W5-W4}
U(n+1) =  U(n) + \frac{r_n}{(n+3)^{1-c}} + \frac{\Delta \Psi(n)}{(n+3)^{1-c}},
\end{equation}
where (when $n\to+\infty$ in the second equality)
\begin{linenomath}
\begin{align*}
r_n &= F_5(\hat{{\bf W}}(n))+\hat W_5(n)-F_4(\hat{{\bf W}}(n))-\hat W_4(n)+\big((n+2)^{1-c}-(n+3)^{1-c}\big)U(n)\\
&= F_5(\hat{{\bf W}}(n))-F_4(\hat{{\bf W}}(n))+\hat W_5(n)-\hat W_4(n)+(c-1+o(1))n^{-c}U(n),
\end{align*}
\end{linenomath}
and $\Delta \Psi(n) := \Delta M_5(n+1) - \Delta M_4(n+1)$.
By Lemma~\ref{lem:seedW5}, for all $\beta>\max(2\alpha-1,\frac \alpha 2)$,
almost surely on $\mathcal A(\alpha)$,
we have $\hat W_4(n)+\hat W_5(n)\leq \hat W_3(n)+ \mathcal O(n^{\beta-1})$ 
when $n\to+\infty$. 
Using Lemma~\ref{lem:F5-F4}{\bf (a)}, this implies that, {almost surely on $\mathcal A(\alpha)$}, for all $n$ large enough,
\[F_5(\hat{{\bf W}}(n))-F_4(\hat{{\bf W}}(n))
\geq (c-\nicefrac14) \big(\hat W_4(n)-\hat W_5(n)\big)-\mathcal O(n^{\beta-1}),\]
and thus (using the fact that $U(n) = (\hat W_5(n)-\hat W_4(n))(n+2)^c$)
\begin{equation}\label{eq:rn}
r_n\geq  (1-c + \nicefrac 14)(\hat W_5(n)-\hat W_4(n))-\mathcal O(n^{\beta-1})+(c-1+o(1))n^{-c}U(n)
= \Big(\nicefrac 14+o(1)\Big)
n^{-c}U(n)-\mathcal O(n^{\beta-1}),
\end{equation}
for all $\beta>\max(2\alpha-1,\frac \alpha 2)$.   
Note that $\Delta U(n) \neq 0$ implies $W_3(n+1)-W_3(n)=1$, and thus by Lemma~\ref{lem:W3} on $\mathcal A(\alpha)$, one has 
\[\mathbb E_n[|\Delta \Psi(n)|^2] \le \mathbb P(\Delta U(n) \neq 0) =\mathcal O(n^{\alpha-1}).\] 
Using next Lemma~\ref{lem:martingale}, we deduce that for any $\beta>\alpha/2$, 
\begin{equation}\label{eq:Psi}
\sup_{n\ge m} \left| \sum_{i=m}^n \frac {\Delta \Psi(i)}{(i+3)^{1-c}}\right| =\mathcal O( m^{-1+c+\beta}).
\end{equation}
By Equation~\eqref{eq:W5-W4}, we have, for all $n> m$,
\begin{equation}\label{eq.U.c}
U(n)= U(m)+\sum_{i=m}^{n-1} \frac{r_i}{(i+3)^{1-c}}+\sum_{i=m}^{n-1} \frac{\Delta \Psi(i)}{(i+3)^{1-c}}.
\end{equation}
Now fix $\beta>\max(2\alpha - 1,\frac \alpha 2)$. 
Observe that if for some $\varepsilon>0$,  $\limsup_{m\to+\infty} (W_5(m)-W_4(m))/m^{\beta+\varepsilon}>0$,  
then Equations \eqref{eq:rn}, \eqref{eq:Psi}, \eqref{eq.U.c} and Lemma~\ref{lem:martingale} imply, by induction, that 
$U(n)\geq  0$, for all $n$ large enough. Thus on $\mathcal A(\alpha)$, there are only two possibilities: either 
$W_5(n) \le W_4(n) +\mathcal O(n^\beta)$, for all $\beta>\max(2\alpha - 1,\frac \alpha 2)$, or $W_5(n) \ge W_4(n)$ for all large enough $n$.

\bigskip

\underline{Second step.} Consider first the case when 
$W_5(n) \le W_4(n) + \mathcal O(n^{\beta})$, for any $\beta>\max(2\alpha - 1,\frac \alpha 2)$. 
Note that, for all $w\in\mathcal E$ and asymptotically when $w_3, w_4, w_5\to 0$, we have, {using Equation \eqref{eq.F2},}
\begin{equation}\label{eq.F5simple}
F_5(w) = \frac{(w_4-w_5)w_2w_5}{w_3+w_1w_4+w_2w_5}
=\frac{(w_4-w_5)w_5}{w_3+w_4+w_5}(1+o(1)).
\end{equation}
By Lemma~\ref{lem:seedW5}, on $\mathcal A(\alpha)$, 
we have $W_3(n)\le W_4(n) + W_5(n)\leq W_3(n)+\mathcal O(n^\beta)$. 
Since we also assume, in this second step, 
that $W_5(n) \le W_4(n) + \mathcal O(n^{\beta})$, 
this implies $2W_5(n)\leq W_3(n)+\mathcal O(n^\beta)$.
Therefore, using that $W_4(n)-W_5(n)\geq -{W_3(n)\wedge}\mathcal O(n^{\beta})$ 
we get
$F_5(\hat{{\bf W}}(n))\geq -\mathcal O(n^{\beta-1})$.
Next, let us prove that for any $\beta>\max( 2\alpha- 1, \frac \alpha 2)$,   
$W_5(n) = \mathcal O(n^{\beta})$. Using \eqref{eq:algo_sto}, we have, for $n\ge m$,
\[
\hat W_5 (n) = \hat W_5(m) + \sum_{i=m}^{n-1}\frac{F_5(\hat{{\bf W}}(i))}{i+3}+ \sum_{i=m}^{n-1}\frac{\Delta M_5(i+1)}{i+3},
\]
where  by Lemma~\ref{lem:martingale} the two sums are greater than $-\mathcal O(m^{\beta-1})$. On $\mathcal A(\alpha)$, if $\limsup_m\hat{W}_5(m)/m^{\beta-1}=\infty$, then the equation above would contradict that $\hat{W}_5(n)$ goes to zero, when $n\to \infty$.
Thus $W_5(n) = \mathcal O(n^\beta)$, as claimed.

Now note that, for all $w\in\mathcal E$, asymptotically when $w_3, w_4, w_5\to 0$,
\[F_3(w)  = p_{135}(w) + p_{234}(w) - w_3 
= \frac{w_3(w_5+\frac{w_4}2)}{w_3 + w_4 + w_5} (1+o(1))- w_3 
\le \frac{2w_5+w_4}{4}(1+o(1)) - w_3,\]
where we have used Equation~\eqref{steak} and the fact that $w_4+w_5\ge w_3$ for all $w\in\mathcal E$.
Since, for all $w\in\mathcal E$, $w_4\leq w_3+w_5$, we get
\[F_3(w)\leq  \frac{3(w_5-w_3)}{4}(1+o(1)).\]
Applying this to $w=\hat{{\bf W}}(n)$ (which belongs to $\mathcal E$ by Lemma~\ref{lem:setE}), we get that on $\mathcal A(\alpha)$, almost surely when $n\to+\infty$,
\[F_3(\hat{{\bf W}}(n)) \leq \mathcal O(n^{\beta-1}).\]
Now if $\liminf {W_3(n)}/{n^\beta} \le 1$, for some $\beta>\max(2\alpha-1, \nicefrac\alpha2)$, 
then Lemma~\ref{lem:martingale}  
gives $W_3(n) = \mathcal O(n^\beta)$ following an argument very similar to the one above for $\hat{\bf W}_5(n)$. 
On the other hand, if $\liminf {W_3(n)}/{n^\beta} =+\infty$, for any 
 $\beta>\max( 2\alpha- 1, \frac \alpha 2)$, 
then, as $W_4(n)\ge W_3(n)-W_5(n)\ge W_3(n)-\mathcal O(n^\beta)$, we have $W_4(n) - W_5(n)\ge 0$, for all $n$ large enough,  which implies $F_5(\hat{{\bf W}}(n)) \ge 0$. 
From \eqref{eq:algo_sto}, this means that for $m$ large enough, the process $(W_5(n))_{n\ge m}$ stochastically dominates a P\'olya urn process $(R_n)_{n\ge m}$ defined by $\mathbb P(R_{n+1} = R_n + 1\mid R_n ) = 1- P(R_{n+1} = R_n \mid R_n )= \frac{R_n}{n+2}$, which is well known to grow almost surely linearly in $n$ (this can be seen using Rubin's construction as in the proof of Lemma~\ref{lem:W5poldec}). Thus $W_5(n)$ would also grow linearly in $n$, and we would get a contradiction. Therefore, necessarily $W_3(n) = \mathcal O(n^\beta)$, as wanted. 

\bigskip 
 
\underline{Third step.} Consider next the case when $W_5(n) \ge W_4(n)$, for all $n$ large enough. 
Define $V(n) = \frac 92 W_4(n) - W_3(n)$, and $\hat V(n) = \frac{V(n)}{n+2}$. One has for any $n\ge 1$, 
\[\hat V(n+1) = \hat V(n) + \frac{H(\hat{{\bf W}}(n))}{n+3} + \frac{\Delta \Theta(n)}{n+3},\]
with again $\Delta \Theta(n)$ the increment of some martingale, and $H(w)  = \frac 92 F_4(w) - F_3(w)$. Using Lemmas~\ref{lem:F5-F4}{\bf (b)} 
and~\ref{lem:martingale} (with arguments similar to those in the second step), 
we deduce that $V(n) \le \mathcal O(n^\beta)$, for any $\beta>\max( 2\alpha- 1, \frac \alpha 2)$. 
We note finally that by Lemma~\ref{lem:F5-F4}{\bf (c)} this entails $F_4(\hat{{\bf W}}(n)) \ge -\mathcal O(n^{\beta{\cec -1}})$, and thus 
by another application of Lemma~\ref{lem:martingale}, we conclude that $W_4(n) = \mathcal O(n^\beta)$, for any $\beta>\max( 2\alpha- 1, \frac \alpha 2)$. 
Then we can use the same argument as in step 2: we first observe that this entails 
\[F_3({\hat{{\bf W}}(n)})
\le \frac{\hat W_5(n)}{2}(1-o(1))-{\hat W_3(n)} + \mathcal O(n^{\beta-1})
\le - \frac{\hat W_3(n)}{4} + \mathcal O(n^{\beta-1})\le  \mathcal O(n^{\beta-1}).\] 
Therefore, if $\liminf \frac{W_3(n)}{n^{\beta}} = \infty$, then $W_3(n) \sim W_5(n)$, and $W_4(n)= o(W_3(n))$. 
Thus by Lemma~\ref{lem:F5-F4}{\bf (c)} again (applied with $r=0$), we get $F_4(\hat{{\bf W}}(n))\ge 0$, 
for all $n$ large enough, which leads to a contradiction as in step 2. 
We conclude that $W_3(n) = \mathcal O(n^\beta)$, as wanted. 
This concludes the proof of the lemma. 
\end{proof}

An immediate corollary of the results obtained so far is the following fact. 
\begin{corollary}\label{cor:final}
On the event when $W_1(n)/n\to 1$, one has almost surely for any $\varepsilon>0$, 
\[W_3(n)= \mathcal O(n^\varepsilon).\] 
\end{corollary} 
\begin{proof} It suffices to combine Lemmas~\ref{lem:W3poldec} and~\ref{lem:W5poldec} 
with Lemma~\ref{lem:polWi}, which we can iterate as much as needed.
Indeed, the map $\varphi : \alpha\mapsto \max(2\alpha -1, \frac \alpha 2)$ is decreasing, with $0$ as unique fixed point in $[0,1)$, 
which implies that any sequence defined by $\alpha_{n+1} = \varphi(\alpha_n)$, with $\alpha_0<1$, converges to~$0$. 
Lemmas~\ref{lem:W3poldec} and~\ref{lem:W5poldec} give the existence of $\alpha_0<1$ such that, on the event $W_1(n)/n\to 1$, $\mathcal A({\alpha_0})$ has probability~1. Lemma~\ref{lem:polWi} then implies that for all $n\geq 0$, on the event $W_1(n)/n\to 1$, $\mathcal A({\alpha_n})$ has probability~1. We then choose $n$ large enough so that $\alpha_n<\varepsilon$.
\end{proof}

The final step is the following result, which together with Corollary~\ref{cor:final} brings a contradiction, if $W_1(n)/n \to 1$, and therefore concludes the proof of Proposition~\ref{prop:nonconv}. 
\begin{lemma}\label{lem:finalstep}
On the event when $W_1(n)/n \to 1$, one has almost surely for any $c\in (0,\nicefrac 15)$, 
\[\lim_{n\to \infty} \frac{W_3(n)}{n^c} =+\infty.\] 
\end{lemma}
\begin{proof}
Recall that when $w_3$, $w_4$ and $w_5$ go to $0$, one has for $w\in \mathcal E$,  
\[p_{135}(w) + p_{234}(w) = \frac{w_3(w_5+ \frac{w_4}{2})}{w_3+w_4+w_5} (1+o(1)). \]
Using now that $w_4+w_5\ge w_3$, we get that
\[p_{135}(w) + p_{234}(w) \geq \frac{w_3(1+o(1))}{4} =  \frac{\nicefrac{w_3}4(1+o(1))}{1-w_3+\nicefrac{w_3}4}.\]
Thus there exists $\varepsilon>0$, such that for any $w\in \mathcal E$, with $w_3,w_4,w_5\le \varepsilon$, 
\[p_{135}(w) + p_{234}(w) \geq \frac{\nicefrac{w_3}5}{1-w_3+\nicefrac{w_3}5}.\]
By Proposition~\ref{prop:W3} and Lemma~\ref{lem:W5poldec}, we deduce that almost surely on the event when $W_1(n)/n\to 1$, 
there exists a random integer $n_0$ such that, for all $n\geq n_0$,
\[p_{135}(\hat {{\bf W}}(n)) + p_{234}(\hat {{\bf W}}(n)) \geq  \frac{\hat W_3(n)/5}{1-\hat W_3(n)+\nicefrac{\hat W_3(n)}5} 
= \frac{\nicefrac{W_3(n)}5}{n+2-W_3(n)+\nicefrac{W_3(n)}5}.\]
Therefore, after some (random) time $n_0$, the process $(W_3(n))_{n\geq n_0}$ stochastically dominates 
an urn process $(U(n))_{n\geq n_0}$, defined 
by $U(n_0) = 1$ and, for all $n\geq n_0$,
\[\mathbb P\Big(U(n+1) = U(n)+1\mid U(n)\Big)= 1-\mathbb P\Big(U(n+1) = U(n)\mid U(n)\Big) = \frac{U(n)/5}{n+n_0+2-U(n)+U(n)/5}.\]
For any fixed $n_0$, the urn process $(U(n))_{n\geq n_0}$ is studied for instance in Janson~\cite{Janson05} (see in particular Theorem 1.4 and Remark 1.12 there), which provides a precise asymptotic behavior of $n^{-\nicefrac 15} U(n)$: it converges in law towards  
some non-degenerate random positive variable. But here one can simply rely again on Rubin's construction, which covers our needs. It shows that for any fixed $n_0$, 
almost surely there exists a constant $c>0$, such that $U(n)\ge cn^{1/5}$, for any $n\ge n_0$,   
and the lemma follows. 
\end{proof}
The proofs of Proposition~\ref{prop:nonconv} and Theorem~\ref{theo:losange} are now complete.  

\section{Proof of Proposition~\ref{prop:ex_Daniel}}\label{sec:ex_Daniel}
We fix an integer~$L$ and look at the uniform-geodesic version of the model 
on the graph on the left-hand side of Figure~\ref{fig:c_ex}. 
Note that each ant reinforces either the $L$ edges on the left (and the edge linked to $F$) 
or the $L$ edges on the right (and the edge linked to $F$). 
{Thus, the $L$ edges on the left have all the same weight at all times, 
and similarly for the the $L$ edges on the right.}
For all integers~$n$, we set $N_1(n)$ to be the weight of the $L$ left-edges at time~$n$, 
and by $N_2(n)$ the weights of the right-edges at time $n$. 
By definition, we have that $N_1(n)+N_2(n) = n+2$, which is also the weight of the edge linked to $F$.

To prove Proposition~\ref{prop:ex_Daniel}, we apply a result of~\cite{HLS} (see~\cite[Theorem~2.8]{Pemantle}):
\begin{theorem}\label{th:HLS}
Let $(Z_n)_{n\geq 0}$ be a sequence of random variables taking values in $[0,1]$ satisfying, for all $n\geq 0$,
\begin{equation}\label{eq:HLS}
Z_{n+1} = Z_n +\frac1n\big(F(Z_n)+\Delta M_{n+1}\big),
\end{equation}
where $F : [0,1]\to [0,1]$ and $\Delta M_{n+1}$ is a martingale increment.
If there exists $\varepsilon>0$ such that $F<0$ on $[0,\varepsilon]$, then $\mathbb P(Z_n\to 0)>0$.
\end{theorem}
First note that, if $Z_n = N_1(n)/(n+2)$ for all $n\geq 0$, then $Z_n$ satisfies Equation~\eqref{eq:HLS} with
$F(x) = p(x)-x$, where $p(x)$ is the probability that an ant reinforces 
the left-hand side geodesic after performing a random walk on $\mathcal G$ 
with weights $x$ on all edges on the left, $1-x$ on all edges on the right, and~1 on the edge linked to $F$.
We let $\mathcal G(x)$ denote the graph $\mathcal G$ equipped with these weights, 
$P$ denote the unique vertex neighbouring~$F$, and for all $k\in\{0, \ldots, L-1\}$, 
$A_k$ denote the vertex at distance~$k$ of $N$ on the right-hand-side geodesic (with $A_0 = N$).
See Figure~\ref{fig:notations} where the notations are illustrated.

\begin{figure}
\begin{center}
\includegraphics[width=2.3cm]{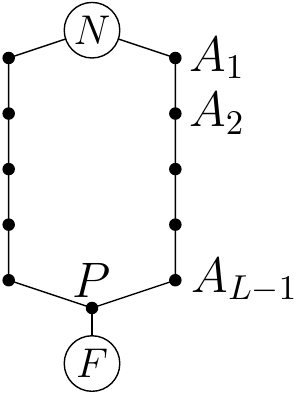}
\end{center}
\caption{The graph of Proposition~\ref{prop:ex_Daniel} and the notations used in Section~\ref{sec:ex_Daniel}.}
\label{fig:notations}
\end{figure}

We now calculate $p(x)$ when $x\to0$ to show that $p(x)<x$ in a neighbourhood of zero; 
this implies that $F<0$ in a neighbourhood of zero and thus that Theorem~\ref{th:HLS} applies. 
 Asymptotically when $x\to0$,
\begin{equation}\label{eq:p(x)}
p(x) = \sum_{k=0}^{L-1} \Big(p^{\sss (1)}_k(x) + \frac12p^{\sss (2)}_k(x)\Big) 
+ \frac12 p^{\sss (3)}(x) + \frac12 p^{\sss (4)}(x) + \mathcal O(x^2),
\end{equation}
where, for all $k\in \{0, \ldots, L-1\}$,
\begin{itemize}
\item $p^{\sss (1)}_k(x)$ the probability that a walker on the weighted graph $\mathcal G(x)$
goes from $N$ to $A_k$ using only edges on the right-hand-side geodesic, 
then goes from {$A_k$ to $N$ without reaching $A_{k+1}$, then goes from $N$ to $P$ without reaching $A_{k+1}$}, 
and, finally, goes from $P$ to $F$ without using the left-hand-side geodesic or reaching $A_k$;
\item $p^{\sss (2)}_k(x)$ the probability that a walker on the weighted graph $\mathcal G(x)$
goes from $N$ to $A_k$ using only edges on the right-hand-side geodesic, 
then goes from {$A_k$ to $P$ without reaching $A_{k+1}$ (thus using edges on the left-hand-side geodesic)}, 
and, finally, goes from $P$ to $A_k$ using only edges on the right-hand side geodesic;
\item $p^{\sss (3)}(x)$ the probability that a walker on the weighted graph $\mathcal G(x)$
first goes from $N$ to $P$ only using edges on the right-hand-side geodesic, 
then goes {from $P$ to $N$ using edges on the left-hand-side geodesic, and before entering the left-hand geodesic from $N$};
\item $p^{\sss (4)}(x)$ the probability that a walker on the weighted graph $\mathcal G(x)$
first goes from $N$ to $P$ only using edges on the right-hand-side geodesic, 
then goes back from $P$ to $N$ only using edges on the right-hand-side geodesic, and,
finally, goes {from $N$ to $P$ using edges on the left-hand-side geodesic and before entering the left-hand-side geodesic from $P$ or hitting $F$}.
\end{itemize}
The $\mathcal O(x^2)$-term in Equation~\eqref{eq:p(x)} stands for all trajectories 
of the walker that leave $N$ or $P$ at least twice towards the left.
We have, if $k\in\{1, \ldots, L-1\}$, 
\[p^{\sss (1)}_k(x)
= \frac{\frac{1-x}k}{\frac{1-x}k+x}\cdot \frac{\frac{1-x}k}{\frac{1-x}k+(1-x)} \cdot \frac{\frac xL}{\frac xL+\frac{(1-x)}{k+1}}\cdot\frac{1}{1+x+\frac{1-x}{L-k}}
= \frac{x}{L}\cdot\frac{L-k}{L-k+1} + \mathcal O(x^2).
\]
We also have
\[p^{\sss (1)}_0(x) = \frac{\frac xL}{\frac xL+(1-x)}\cdot\frac{1}{1+x+\frac{1-x}{L}}
=\frac{x}{L}\cdot \frac{L}{L+1}+\mathcal O(x^2).
\]
Using the fact that $\sum_{i=1}^n \nicefrac1i = \log n+ \mathcal O(1)$ when $n\to+\infty$, we get 
\begin{equation}\label{eq:p1}
\sum_{k=0}^{L-1} p_k^{\sss (1)} (x)
= \frac xL\sum_{k=0}^{L-1} \Big(1-\frac{1}{L-k+1}\Big) + \mathcal O(x^2)
= x\left(1-\frac{\log L}{L} + \mathcal O_{L\to+\infty}(1)\right) + \mathcal O(x^2),
\end{equation}
where the $\mathcal O_{L\to+\infty}(1)$-term does not depend on $x$ and corresponds to the $L\to+\infty$ limit, 
while the $\mathcal O(x^2)$-term depends on $L$ and refers to the $x\to0$ limit.
Similarly, for all $k\in\{1, \ldots, L-1\}$, we have 
\[p_k^{\sss (2)}(x)
=\frac{\frac{1-x}k}{\frac{1-x}k+x}\cdot\frac{\frac{1-x}k}{\frac{1-x}k+1-x}\cdot \frac{\frac xL}{\frac xL+\frac{1-x}{k+1}}\cdot \frac{\frac{1-x}{L-k}}{\frac{1-x}{L-k}+1+x}
= \frac{x}{L}\cdot \frac1{L-k+1},
\]
and 
\[p_0^{\sss (2)}(x)
=\frac{\frac xL}{\frac xL+(1-x)}\cdot \frac{\frac{1-x}{L}}{\frac{1-x}{L}+1+x}
=\frac xL\cdot \frac1{L+1}+\mathcal O(x^2).
\]
Using again the asymptotic behaviour of the harmonic sum, we get 
\begin{equation}\label{eq:p2}
\frac12\sum_{k=0}^{L-1} p_k^{\sss (1)}(x)
=\frac xL\sum_{k=0}^{L-1}\frac1{L-k+1}+\mathcal O(x^2)
=x\left(\frac{\log L}{2L}(1+o_{L\to+\infty}(1))\right)  + \mathcal O(x^2).
\end{equation}
We also have, when $x\to0$, 
\[p^{\sss (3)}(x) = \frac{\frac{1-x}L}{\frac{1-x}L+x}\cdot \frac{\frac xL}{\frac xL+1+\frac{x\frac{1-x}{L}}{x+\frac{1-x}{L}}}
= \frac x{L} + \mathcal O(x^2),\]
and
\[p^{\sss (4)}(x)= \frac{\frac{1-x}L}{\frac{1-x}L+x}\cdot \frac{\frac{1-x}L}{\frac{1-x}L+1+x}\cdot \frac{\frac xL}{\frac xL+\frac{(1+x)\frac{1-x}{L}}{1+x+\frac{1-x}{L}}}
= \frac x{L}+\mathcal O(x^2).
\]
Using these last equations together with~\eqref{eq:p1} and~\eqref{eq:p2} 
into Equation~\eqref{eq:p(x)}, we get that, in total, 
\[p(x) =x\left(1-\frac{\log L}{2L} (1+ o_{L\to+\infty}(1))\right)+\mathcal O(x^2).\]
Therefore,
\[F(x) = p(x)-x =-x\left(\frac{\log L}{2L}(1+ o_{L\to+\infty}(1))\right)+\mathcal O(x^2),\]
implying that for all $L$ large enough, $F$ is indeed negative in a right-neighbourhood of~$0$.
Hence, Theorem~\ref{th:HLS} applies and we conclude that $\mathbb P(Z_n\to 0)>0$,
which concludes the proof of Proposition~\ref{prop:ex_Daniel}.

\bibliographystyle{alpha}
\bibliography{fourmis_arxiv}
\end{document}